\documentclass[12pt]{article}

\usepackage{amsmath,amssymb,amsthm,graphicx}
\usepackage{a4wide}

\newtheorem{lemma}{Lemma}
\newtheorem{theorem}[lemma]{Theorem}

\newtheorem{prop}[lemma]{Proposition}
\newtheorem{coro}[lemma]{Corollary}

\numberwithin{lemma}{section}
\numberwithin{fact}{section}
\numberwithin{equation}{section}

\title{Exact solution of two classes of prudent polygons}

\author{Uwe~Schwerdtfeger\\
\\
Fakult\"at f\"ur Mathematik\\
Universit\"at Bielefeld\\
Postfach 10 01 31, 33501 Bielefeld, Germany\\
}

\begin{document}

\maketitle

%
%
%
\begin{abstract}Prudent walks are self-avoiding walks on a lattice which never step into the direction of an already occupied vertex. We study the closed version of these walks, called prudent polygons, where the last vertex of the walk is adjacent to its first one. More precisely, we give the half-perimeter generating functions of two subclasses of prudent polygons on the square lattice, which turn out to be algebraic and non-D-finite, respectively. 
\end{abstract}
\section{Introduction}
The enumeration of self-avoiding walks (SAW) and polygons (SAP) on a lattice by their number of steps \cite{MaSl} is a long standing problem in combinatorics. Extrapolation of series data from exact enumeration has led to high precision estimates of the exponential growth rate and subexponential corrections but an exact solution of either problem (i.e. finding the generating function, see below) seems out of reach. 
Rechnitzer \cite{AR} has shown that the anisotropic generating function of SAPs on the square lattice is not D-finite. A (possibly multivariate) function $f(\textbf{z})$ is \emph{D-finite}, if the vector space over $\mathbb{C}(\textbf{z})$ spanned by its derivatives is finite dimensional. In the univariate case this means that $f$ is a solution of a homogeneous linear ordinary differential equation with polynomial coefficients. At present one tries to find solvable subclasses with large exponential growth rates. This approach is particularly successful in two dimensions. We will restrict to the square lattice in our paper. 

One promising class are so-called \emph{prudent walks} (PW) \cite{Duc,Pre}, which are SAWs never stepping towards an already occupied vertex. Note that a general prudent walk is not \emph{reversible}, i.e. the walk traversed backwards from its terminal vertex to its initial vertex may not be prudent. Since SAWs are counted modulo translation,  we may choose the initial vertex of a PW to be the origin $(0,0).$ The full problem of PW is unsolved, but recently Bousquet-M\'elou \cite{Bou2} succeeded in enumerating a substantial subclass. We adopt the terminology of her paper and use the same methods to obtain the generating functions for the corresponding polygon models defined below. Every nearest neighbour walk on the square lattice has a minimal bounding rectangle containing it, referred to as the \emph{box} of the walk. It is easy to see that each unit step of a prudent walk ends on the boundary of its current box. (This is \emph{not} a characterisation of PWs, e.g. the walk $(0,0)\to (0,1)\to (1,1)\to (2,1)\to (2,0)\to (1,0)$ is not prudent.) This property allows the definition of the following subclasses. Call a PW \emph{one-sided}, if every step ends on the \emph{top} side and \emph{two-sided}, if every step ends on \emph{top} or on the \emph{right} side of the current box. Similarly, a PW is referred to as \emph{three-sided} if every step ends on the \emph{left}, \emph{top} or the \emph{right} side of the current box and additionally each \emph{left step} and each \emph{right step} that ends on the \emph{bottom side} of the  current box \emph{inflates} the box.\\ 
\textbf{Remark.} \textit{i)} As soon as the width of the box of a PW is greater than one, the latter additional condition is redundant. It rules out certain configurations which can occur only if the box width is equal to one, namely ``downward zig-zags" of width one, e.g. $\ldots\to (1,0)\to (1,-1)\to (0,-1)\to (0,-2) \to (1,-2)\to (1,-3) \to \ldots$ which needlessly complicate the computations below.\\
\textit{ii)} Duchi \cite{Duc} introduced two-sided and three-sided PWs as \emph{type-1} and \emph{type-2} PWs, respectively. In \cite{DGJ,GGJD09} the authors also employ her notation.

\smallskip
Explicit expressions for the generating functions of one-, two- and three-sided PWs have been found so far confirming the data obtained in \cite{DGJ,GGJD09} by computer enumeration.  
The first class consists of partially directed walks and has a rational generating function. The second class was shown to have an algebraic generating function by Duchi \cite{Duc} and recently in \cite{Bou2} the third class was solved and the generating function was found not to be D-finite. 

\smallskip
Guttmann \cite{DGJ,GGJD09,AJG} proposed to study the polygon version of the problem, meaning walks, whose last vertex is adjacent to the starting vertex. We exclude single edges from this definition. As above, the property of being prudent demands a starting vertex and a terminal vertex. So prudent polygons are \emph{rooted} polygons with a directed root edge. Note further that a prudent polygon (PP) which ends, say, to the right of the origin (i.e. in the vertex $(1,0)$) may never step right of the line $x=1,$ and furthermore if the walk hits that line it has to head directly to the vertex $(1,0).$ So prudent polygons are \emph{directed} in the sense that they contain a corner of their box. Moreover, a $k$-sided PP can be interpreted as a $(k-1)$-sided PW confined in a half-plane, see also Section \ref{Conclusion}. In this paper we deal with the polygon versions of the two-sided and three-sided walks, referred to as two-sided and three-sided PPs. Enumeration of one-sided PPs is trivial, since these are simply rows of unit cells. We give explicit expressions for the half-perimeter generating functions of two-sided PPs and three-sided PPs and show that the latter is not D-finite, as expected on numerical grounds \cite{DGJ,GGJD09}. To our knowledge three-sided PPs are the first \emph{exactly solved} polygon model with a non-D-finite \emph{half-perimeter} generating function. Enumeration of the full class of PPs remains an open problem, as for the walk case. 

\smallskip
\noindent
\textbf{Outline:} In Section \ref{Functionalequations} we give functional equations for the generating functions which are based on decompositions of the classes in question, in Section \ref{Solution} we solve those by the kernel method \cite{Bou1,Bou2,BouJeh,MisRec} and in Section \ref{Analyticproperties} we study the analytic behaviour of the generating functions of two-sided and three-sided prudent polygons. Section \ref{Random} is dedicated to the random generation of PPs.
\section{Functional equations}\label{Functionalequations}
In combinatorial enumeration of objects from a class $\cal P$ (say PPs) with respect to counting parameters $c_1,\ldots,c_n:\; {\cal P}\to \mathbb{Z}_{\geq 0}$ (say perimeter, area etc.) the (multivariate) power series
\[
P\left(x_1,\ldots,x_n\right)=\sum_{\alpha\in {\cal P}} x_1^{c_1(\alpha)}\cdots x_n^{c_n(\alpha)}
\]
is called the \emph{generating function} of $\cal P.$
In the following we will count prudent polygons by half-perimeter and other, so-called \emph{catalytic counting parameters}. The variables in the generating function marking the latter are called \emph{catalytic variables}. Their introduction allows us to translate certain combinatorial decompositions into non-trivial functional equations for the associated generating functions \cite{GouJac}. Furthermore, we identify a PP (a ``closed" PW) with the collection of unit cells it encloses and build larger PPs from smaller ones by attaching unit cells in a prudent fashion, i.e. the new boundary walk with the same initial vertex remains prudent.

\smallskip
A two-sided prudent polygon either ends at the vertex above the origin or at the vertex to the right of it. This partitions two-sided PPs into two subsets, which can be transferred into each other by the reflection in the diagonal $x=y.$ So it suffices to enumerate prudent polygons ending on the top of their box. Here two cases occur, namely the ``degenerate" case when the first steps of the walk are left. The resulting PP is simply a row of unit cells pointing to the left. These have a half-perimeter generating function $t^2+t^3+\ldots=t^2/(1-t).$ In the ``generic case" such a PP is a bar graph turned upside down, i.e. a column convex polyomino containing the top side of its bounding box, cf. Figure \ref{sdtop}.  
\begin{figure}[htb]
\begin{center}
\includegraphics[height=14 mm,width=56 mm]{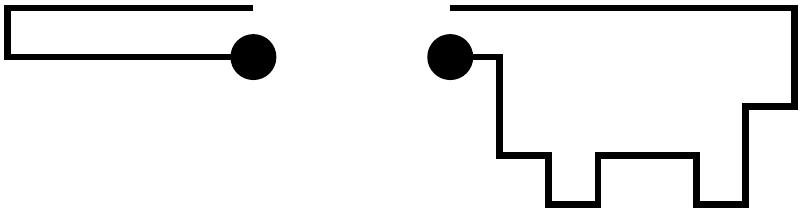}
\caption{\small Degenerate (left) and generic 2-sided PPs ending on the top of the box}\label{sdtop}
\end{center}
\end{figure}
Denote by $B(t,u,w)$ the generating function of bar graphs counted by half-perimeter, width and height of the rightmost column (catalytic parameter), marked by $t,$ $u$ and $w$ respectively. Here $w$ is the catalytic variable. The width parameter is not a catalytic parameter. However, it will be important in the study of three-sided PPs. We follow the lines of \cite{Bou1}. 
\begin{lemma}
The generating function $B(t,u,w)$ of bar graphs satisfies the functional equation
\begin{equation} \label{feqB}
B \left(t,u, w \right)=u\left(\frac {{t}^{2}w}{1-wt}+\frac{wt \left( B \left(t,u, 1 \right) -B \left(t,u, w \right)  \right) }{1-w}+\frac{B \left(t,u, w \right) {t}^{2}w}{1-wt}\right).
\end{equation}
\end{lemma}
\begin{proof}
A bar graph is either a single column, or it is obtained by attaching a new column to the right side of a bar graph. The decomposition is sketched in Figure \ref{feq2sd}. 
\begin{figure}[htb]
\begin{center}
\includegraphics[height=16 mm,width=78mm]{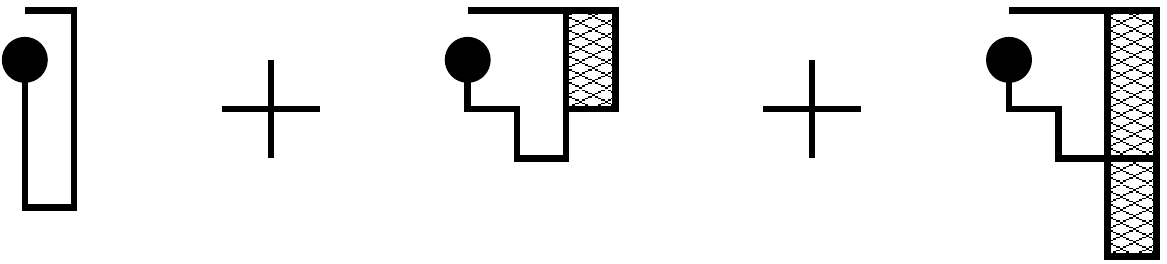}
\caption{\small Illustration of the decomposition underlying functional equation \eqref{feqB}}\label{feq2sd}
\end{center}
\end{figure}

\noindent
Single columns of height $\geq 1$ contribute $u{t}^{2}w/(1-wt)$ to the generating function. The polygons obtained by adding a column which is shorter than or equal to the old rightmost column contribute the second summand. This is seen as follows. A polygon of half-perimeter $n,$ width $k$ and rightmost column height $l$ contributing $t^nu^kw^l$ to $B(t,u,w)$ gives rise to $l$ polygons whose rightmost column is shorter or equal. Their contribution sums up to
\begin{equation}\label{Ableitung}
tu\sum_{j=1}^{l}t^nu^kw^j=tuw\,\frac{t^nu^k1^l-t^nu^kw^l}{1-w}.
\end{equation}
Summing this over all polygons gives the second summand. The third summand corresponds to adding a larger column. To this end duplicate the rightmost column and attach a non-empty column below the so obtained new rightmost column. A so obtained bar graph can be viewed as an ordered pair of a bar graph and a column. The generating function of those pairs is the third summand of the rhs. This finishes the proof.
\end{proof}
\noindent
The walk constituting the boundary of a three-sided PP has $(0,0)$ as its initial vertex and $(1,0)$ or $(-1,0)$ or $(0,1)$ as its terminal vertex. Those walks with terminal vertex $(0,1)$ may not step above the line $y=1$ and they have to move directly to the vertex $(0,1)$ as soon as they step upon that line. This leads to two sorts of bar graphs either rooted on their left or on their right side, see Figure \ref{3sdTop}.
\begin{figure}[htb]
\begin{center}
\includegraphics[height=10 mm,width=60mm]{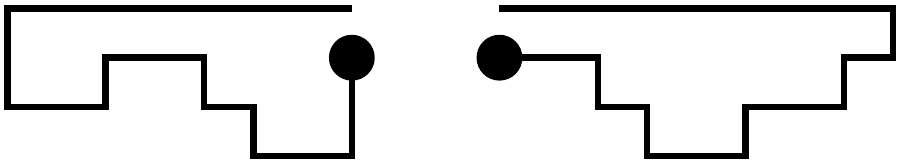}
\caption{\small Three-sided PPs with terminal vertex $(0,1)$ are bar graphs}\label{3sdTop}
\end{center}
\end{figure}

\noindent
So only those three-sided PPs are of further interest, which end in $(1,0)$ or $(-1,0).$ Both classes are transformed into each other by a reflection in the line $x=0.$ We study those ending to the right of the origin in the vertex $(1,0).$ Again a degenerate and a generic case are distinguished, according to whether such a PP reaches its terminal vertex from below via the vertex $(1,-1)$ (``counterclockwise around the origin") or from above, via the vertex $(1,1)$ (``clockwise"). In the degenerate case we simply obtain a single column. In the generic case, a (possibly empty) sequence of initial down steps is followed by a left step.  So denote by $R(t,u,w)$ the generating function of the generic three-sided PPs ending in the vertex $(1,0)$ counted by half-perimeter, the length of the top row and the distance of the top left corner of the top row and the top left corner of the box, marked by $t,$ $u,$ and $w,$ respectively, cf. Figure \ref{3sdcatvar}.  Here both $u$ and $w$ are catalytic variables. 
\begin{figure}[htb]
\begin{center}
\includegraphics[height=20 mm,width=80mm]{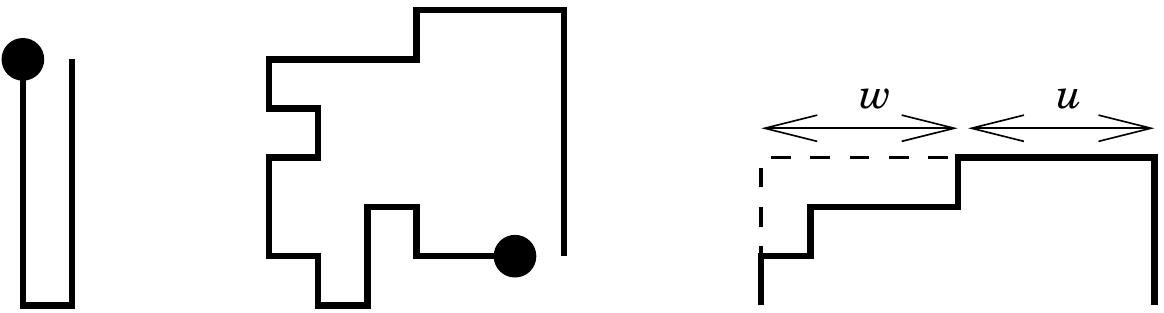}
\caption{\small Degenerate and generic three-sided PPs, catalytic variables}\label{3sdcatvar}
\end{center}
\end{figure}
\begin{lemma}
The generating function $R(t,u,w)$ of generic three-sided PPs satisfies the functional equation 
\begin{equation}\label{feqR}
\begin{split}
R(t,u,w)=\;&ut \left( B \left( t,u \right)+t\right) +\frac {ut \left( R \left( t,w,w \right) -R \left( t,u,w \right)  \right) }{w-u} \\
&+ {\frac {u{t}^{2} \left( R \left( t,u,w \right) -R \left( t,u,ut \right)  \right) }{w-ut}} +R \left( t,u,ut \right) ut \left( B \left( t,u \right) +t \right),
\end{split}
\end{equation}
where $B(t,u):=B(t,u,1)$ is the generating function of bar graphs counted by half-perimeter and width.
\end{lemma}
\begin{proof}
The decomposition we use is sketched in Figure \ref{3sdfeq}. 
\begin{figure}[htb]
\begin{center}
\includegraphics[height=15 mm,width=140mm]{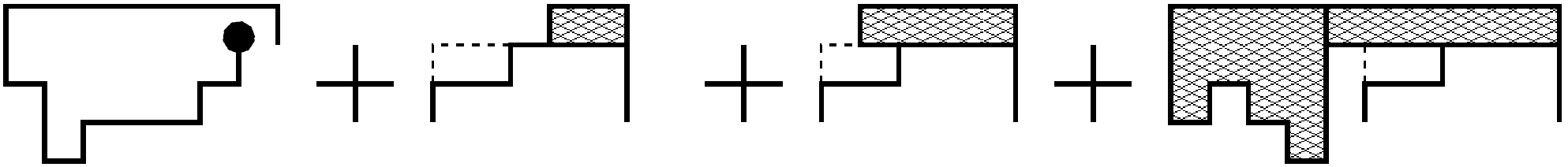}
\caption{\small Illustration of the decomposition underlying functional equation \eqref{feqR}}\label{3sdfeq}
\end{center}
\end{figure}

\noindent
The polygons in question contain the top right corner of their box. This corner is some point $(1,y).$ If $y=1,$ then the PP is either the unit square containing $(0,0)$ and $(1,1)$ or a bar graph as above with that unit square glued to the right. This yields the first summand. A PP with $y>1$ is obtained in one of the following three ways from a PP with top right corner $(1,y-1).$ The first is to add a new row on top, which is shorter than or equal to the original top row. A similar computation as in \eqref{Ableitung} (with some additional book keeping on $w$) yields the second summand.  The second way to obtain a larger PP from a smaller one is by adding a new row on top, which is longer than the original top row, but does not inflate the box to the left. Again a treatment similar to the computation in \eqref{Ableitung} yields the third summand. The third way to extend a PP is to add a row on top of length equal to the width of the box plus one and possibly an arbitrary bar graph. This finally yields the fourth summand and the functional equation is complete. 
\end{proof}
\noindent
\textbf{Remark.} As in the case of general SAPs \cite{Ham61} we can define a concatenation of two three-sided PPs. Roughly speaking, the one PP can be enlargened by inserting the other one at the top corner of the leftmost column, see Figure \ref{concatenation}.\\
\begin{figure}[htb]
\begin{center}
\includegraphics[height=20 mm,width=60mm]{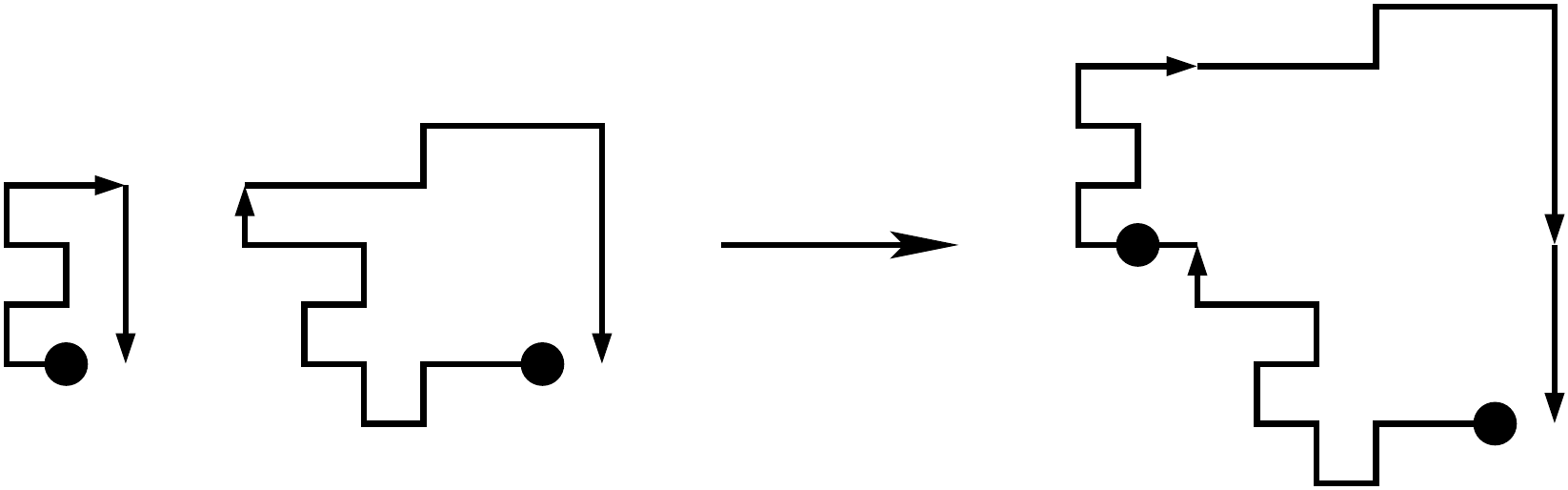}
\caption{\small Concatenating two 3-sided PPs}\label{concatenation}
\end{center}
\end{figure}

\noindent
The numbers $pp_3^{(m)}$ of three-sided PPs hence satisfy $pp_3^{(m+n)}\geq pp_3^{(m)}\cdot pp_3^{(n)}.$ This implies the existence of a connective constant $\beta,$ i.e. a representation $pp_3^{(m)}=\exp(\beta m+o(m)).$ The precise value for $\beta$ and the subexponential corrections are given in Section \ref{Analyticproperties}.    
The converse inequality holds for prudent walks, since breaking an $m+n$ step PW after $m$ steps leaves one with a pair of prudent walks of respective lengths $m$ and $n.$ 

\smallskip
We now turn to unrestricted PPs. They can be partitioned into eight subclasses according to their end point being $(1,0),$ $(0,1),$ $(-1,0)$ or $(0,-1)$ and their orientation (clockwise or counterclockwise around the origin). All eight classes can be transformed into each other by symmetry operations of the square. Hence it suffices to enumerate those PPs ending in $(1,0)$ which reach their endpoint via the vertex $(1,1)$ (clockwise). We denote this class by $\mathcal F$ and by $F(u,w,x):=F(t,u,w,x)$ its generating function. $G(u,w,x):=G(t,u,w,x)$ and $H(u,w,x):=H(t,u,w,x)$ are defined as the generating functions of the two auxiliary subclasses $\mathcal G \supseteq \mathcal H$ of $\mathcal F$ specified below. We have the following functional equation.
\begin{lemma}
The power series $F$, $G$ and $H$ satisfy a system of functional equations. For $X=F,G,H$ the single equations are of the form
\begin{equation}\label{feq4sd}
\begin{split}
X(u,w,x)=\;&\frac {tux \left( X \left( w,w,x \right) -X \left( u,w,x \right)  \right) }{w-u}\\ 
&+ {\frac {{t}^{2}ux \left( X \left( u,w,x \right) -X \left( u,ut,x \right)  \right) }{w-ut}} +I_X(u,w,x),
\end{split}
\end{equation}
where the formal power series $I_X(u,w,x):=I_X(t,u,w,x)$ is equal to
\begin{equation*}
I_X(u,w,x):=\begin{cases}G(x,x,u) &\textnormal{if } $X=F,$\\ t^2ux+t^2uxF(x,xt,w)+xH(x,x,u) &\textnormal{if } $X=G,$  \\ t^2uxw^{-1}G(x,xt,w) &\textnormal{if } $X=H.$ \end{cases}
\end{equation*}
\end{lemma} 
\begin{proof}
The proof relies on a decomposition similar to that of three-sided PPs. 
The PPs of the class $\mathcal F$ all contain the top right corner of their box. In the generating function $F(u,w,x)$ of $\mathcal F$ the variable $u$ marks the length of the top row, $w$ marks the distance of the top left corner of the top row to that of the box and $x$ marks the height of the box.

We define the classes $\mathcal G$ and $\mathcal H.$ $\mathcal G$ consists of the unit square together with those PPs in $\mathcal F$ which are obtained by attaching a piece (a collection of unit cells) on top of a given PP in $\mathcal F,$ such that the top side of the box is shifted by one unit, the left side by at least one unit and the bottom side by an arbitrary number of units, see Figure \ref{4sd}.
\begin{figure}[htb]
\begin{center}
\includegraphics[height=35 mm,width=115mm]{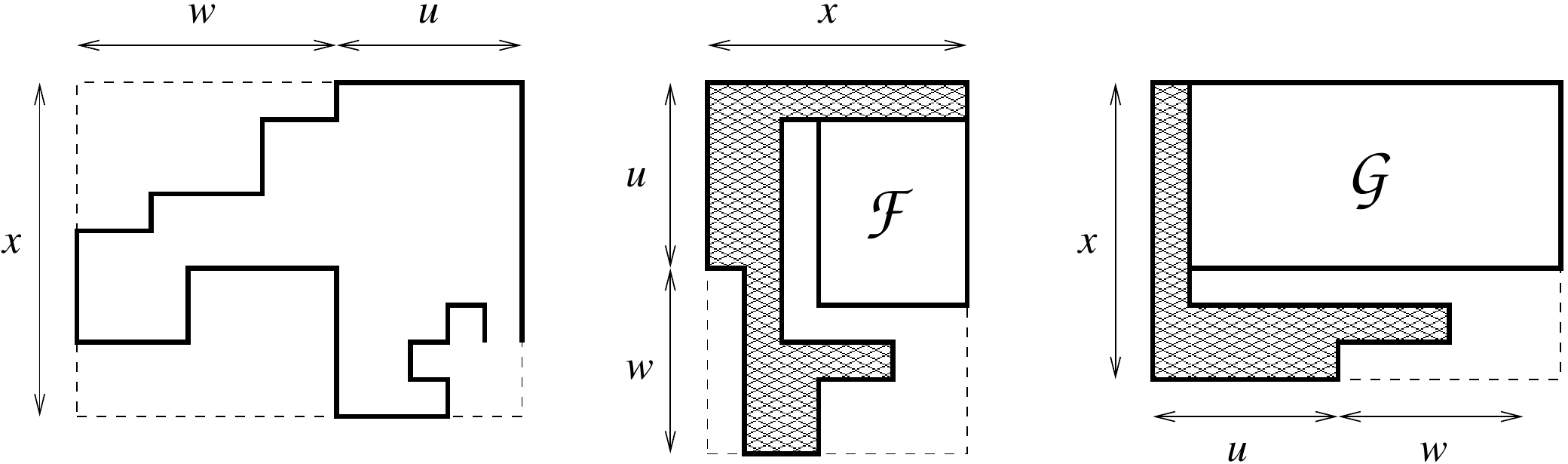}
\caption{\small Illustrations of the classes $\mathcal F,$ $\mathcal G$  and $\mathcal H,$ catalytic variables}\label{4sd}
\end{center}
\end{figure}

\noindent
Note, that the polygons in $\mathcal G$ contain the top left corner of their box. The catalytic variable $u$ in the generating function $G(u,w,x)$ marks the length of the leftmost column, $w$ marks the distance of the lower left corner of the leftmost column to the bottom left corner of the box and $x$ marks the width of the box. The class $\mathcal H$ is the subclass of $\mathcal G$ obtained by glueing a piece to the left of the leftmost column and thereby shifting the left side of the box by exactly one unit and the bottom side by at least one unit, see Figure \ref{4sd}. The variable $u$ in $H(u,w,x)$ marks the length of the bottom-most row, $w$ marks the distance minus one of the lower right corner of that row to the bottom right corner of the box. The variable $x$ marks the height of the box.

Now the functional equations are derived similarly to the three-sided case. Every PP in $\mathcal F$ can be extended by adding a new row on top which is shorter than or equal to the old top row or longer than the old top row, but does not inflate the box to the left. These two operations contribute the first and the second summand in the equation for $X=F,$ as in the proof of equation \eqref{feqR}. Inflating the box to the left yields a PP in $\mathcal G,$ explaining the expression for $I_{F}.$
 
As for the functional equation $G(u,w,x)$ of the  class $\mathcal G$ the first two summands on the rhs correspond to adding a new column to the left, the expression $xH(x,x,u)$ to adding a piece which shifts the bottom boundary of the box, in an analogous fashion as above. The unit square contributes $t^2ux,$ the term $t^2uxF(x,xt,w)$ corresponds to the ``minimal" polygons in $\mathcal G$ obtained by adding a top row on an arbitrary PP of length equal to the width of the box plus one.

The minimal PPs in $\mathcal H$ are those obtained by extending a PP in $\mathcal G$ adding a column to the left of length equal to the height of the box plus one. This explains the term for $I_H$ The rest of the rhs corresponds to adding a new bottom row.
\end{proof}
\section{Solution by the kernel method}\label{Solution}
The following result has already been obtained in \cite{PreBra} as the solution of an algebraic equation which arises from a different decomposition of the class. We derive it here for completeness and to recall the ``classic" kernel method as applied in \cite{Bou1}.
\begin{theorem}\label{bargraph}
The generating function $B(t,u):=B(t,u,1)$ of bar graphs counted by half-perimeter and width is equal to
\begin{equation}\label{gfB}
B(t,u)=\frac{1-t-u(1+t)t -\sqrt{t^2(1-t)^2u^2-2t\left(1-t^2\right)u+(1-t)^2}}{2tu}.
\end{equation}
\end{theorem}
\begin{proof}
The functional equation \ref{feqB} is equivalent to
\begin{equation}\label{kfB}
\begin{split}
0= \left(t^2uw(1-w)-uwt(1-wt)-(1-w)(1-wt) \right)B(t,u,w) \\  
+tuw(1-wt) B(t,u,1) +t^2uw(1-w)
\end{split}
\end{equation}
The kernel equation
\[
0= \left(t^2uw(1-w)-uwt(1-wt)-(1-w)(1-wt) \right)
\]
is a quadratic equation in the catalytic variable $w$ and has the following unique \emph{power series} solution $q(t,u)$
\begin{equation}\label{kernelsolB}
q(t,u)=\frac{1+(1-u)t+ut^2-\sqrt{t^2(1-t)^2u^2-2t\left(1-t^2\right)u+(1-t)^2}}{2t}.
\end{equation}
Upon substituting $w=q(t,u)$ into Eq. \eqref{kfB}, the terms with $B(t,u,w)$ are cancelled and we  can solve for $B(t,u,1),$ which leads to \eqref{gfB}. 
\end{proof}
\noindent
\textbf{Remark.} In principle, $B(t,u,w)$ can also be computed, by substituting the result for $B(t,u,1)$ into Eq. \eqref{kfB}.

\smallskip
By setting $u=w=1$ in the bar graph generating function, adding the contribution of the degenerate two-sided PPs and multiplication by 2, we obtain the following result so far conjectured by series extrapolation from exact enumeration data \cite{DGJ}.
\begin{coro}
The generating function of two-sided prudent polygons is equal to
\begin{equation}\label{gfP}
\begin{split}
PP_2(t)&=\frac{1}{t}\left(\frac{1-3t+t^2+3t^3}{1-t}-\sqrt{(1-t)\left(1-3t-t^2-t^3\right)}\right)\\
&=4~z ^{2}+6~z ^{3}+12~z ^{4}+28~z ^{5}+72~z ^{6}+196~z ^{7}+552~z ^{8}+1590~z ^{9}\\
&+4656~z ^{10}+13812~z ^{11}+41412~z ^{12}+125286~z ^{13}+381976~z ^{14}\\&+1172440~z ^{15}+3620024~z ^{16}+11235830~z ^{17}+35036928~z ^{18}\\
&+109715014
~z ^{19}+344863872~z ^{20}+\ldots
\end{split}
\end{equation}
\end{coro}
Now we turn to the three-sided case. Note that the sum of the catalytic counting parameters, namely the length of the top row and the distance of its top left corner to the top left corner of the box, is equal to the width of the polygon. We have the following result for the generic three-sided PPs ending on the right. It is derived in a similar way as the corresponding result on PWs in \cite{Bou2}.
\begin{theorem}The functional equation \ref{feqR} has a unique power series solution. For the generating function $R(t,w,w)$ of generic three-sided prudent polygons ending on the right and counted by half-perimeter and width we have an explicit expression as a an infinite sum of formal power series  
\begin{equation}\label{gfR}
R(t,w,w)=\sum_{k\ge 0}L\left(\left(tq^2  \right)^kw\right)\prod_{j=0}^{k-1}K\left(\left(tq^2  \right)^{j}w\right).
\end{equation}
Here   
\begin{equation}\label{qR}
q:=q(t,1)={\frac {{t}^{2}+1-\sqrt {1-4\,t+2\,{t}^{2}+{t}^{4}}}{2t}},
\end{equation}
with $q(t,u)$ as in \eqref{kernelsolB} in the proof of Theorem \ref{bargraph}. $K$ and $L$ are given by 
\begin{equation}\label{K}
K(w)=\frac { \left( 1-t \right) q-1- \left(  \left( 1-t+{t}^{2} \right) q -1\right)  \left( B \left( t,qw \right) +t \right) w}{1-t \left( 1+t \right) q- \left(  t\left(1-{t} -{t}^{3} \right) q+{t}^{2} \right)  \left( B \left( t,qw \right) +t \right) w}
\end{equation}
and
\begin{equation}\label{L}
L(w)={\frac { \left( 1+{t}^{2}- \left( 1-2\,t+2\,{t}^{2}+{t}^{4} \right) q \right) 
 \left( B \left( t,qw \right) +t \right) w}{1-t \left( 1+t \right) q- \left( t \left( 1-{t} -{t}^{3} \right) q+{t}^{2} \right)  \left( B \left( t,qw \right) +t \right) w}},
\end{equation}
where $B(t,u)$ is the generating function of bar graphs as in \eqref{gfB}. 
\end{theorem}
\begin{proof}
The functional equation \eqref{feqR} is equivalent to 
\begin{equation}\label{kernelformR}
\begin{split}
0= \left(ut^2(w-u)-ut(w-ut)-(w-u)(w-ut)\right)R(t,u,w)  \\
+\left(ut\left(B(t,u)+t \right)(w-u)(w-ut)-t^2u(w-u)    \right)R(t,u,ut)\\
+ut(w-ut)R(t,w,w)\\
+ut(w-u)(w-ut)\left(B(t,u)+t \right).
\end{split}
\end{equation}
We first solve the kernel equation 
\[
\left(ut^2(w-u)-ut(w-ut)-(w-u)(w-ut)\right)=0
\]
for $u$ and $w.$ The unique power series solutions are $U(t,w)=q(t)w$ resp. $W(t,u)=q(t)tu,$ with $q(t)$ as in \eqref{qR}. We substitute $w=W(t,u)$ in Eq. \eqref{kernelformR} and  obtain an expression for $R(t,u,ut)$ in terms of $R \left( t,qtu,qtu \right),$ namely
\begin{equation}
R \left(t, u,ut \right) =\frac { \left( q-1 \right) R \left( t,qtu,qtu \right) + \left( 1-q \right) \left(1-tq\right) \left( B \left( t,u \right) +t \right) u}{1-qt- \left( 1-q \right) \left(1-tq\right) \left( B \left( t,u \right) +t \right) u}.
\end{equation}
Substitute this into Eq. \eqref{kernelformR} and set $u=U(t,w).$ This relates $R(t,w,w)$ and  $R\left(t, wtq^2,wtq^2\right)$ as follows:
\begin{equation}\label{iter}
R(t,w,w) = K(w)\cdot R\left(t, wtq^2,wtq^2\right) + L(w), 
\end{equation}
with $K(w)$ and $L(w)$ as in \eqref{K} and \eqref{L}. $K(w)$ and $L(w)$ are a formal power series in $t,$ which is seen as follows: $B\left(t,qw\right)$ is well-defined as a formal power series in $t$ as $\left[t^N\right]B(t,u)$ is a polynomial in $u$ of degree at most $N-1.$ Furthermore by the definition of $B$ we see $B(t,u)=t^2u+O(t^3).$ The denominator is now easily seen to be $1+O(t),$ so both $K(w)$ and $L(w)$ are well-defined as formal power series in $t.$
Inspecting the first few coefficients we see $(1-t)q-1=O\left(t^3\right)$ and $1- \left( 1-t+{t}^{2} \right) q=O\left(t^2\right),$ so the numerator of $K(w)$ is $O\left(t^3\right).$ In a similar way the numerator of $L(w)$ is seen to be $w\cdot O\left(t^2\right).$ Moreover we have $tq^2=t+O\left( t^2\right).$ So we can iterate Eq. \eqref{iter} and obtain formula \eqref{gfR}. 
\end{proof}%
\noindent
\textbf{Remark.} \textit{i)} We have the following alternative expressions for $K(w)$ and $L(w):$
\begin{equation}\label{Kalt}
K(w)=\frac{\left((1-q)(1-qt)qwt\left(B(t,qw)+t\right) +t^2q(q-1)\right)(q-1)  }
{ q(1-qt)^2\left((1-q)qwt\left(B(t,qw)+t\right) +t \right)   }
\end{equation}
and
\begin{equation}\label{Lalt}
L(w)=\frac{(1-qt)(1-q^2t)(q-1)qtw\left(B(t,qw)+t\right)  }{ q(1-qt)^2\left((1-q)qwt\left(B(t,qw)+t\right) +t \right) }.
\end{equation}
The expressions \eqref{K} and \eqref{L} were obtained by expressing powers of $q$ in terms of $q,$ e.g.
\begin{equation*}
\begin{split}
q^2=\left(t\left(t^2+1\right)q-t\right)/t^2,\\
q^3=\left(t\left(t^4+2t^2-t+1\right)q-t^3-t\right)/t^3,\\
q^4=\left(tq\left(t^6+3t^4-2t^3+3t^2-2t+1\right)-t+t^2-2t^3-t^5\right)/t^4.
\end{split}
\end{equation*}
\noindent
\textit{ii)} In principle one could also compute $R(t,u,w).$ To obtain the generating function of all three-sided PPs we sum up the contributions of the degenerate PPs and those ending on top, multiply by two and obtain
\[
PP_{3}(t)=2\left(\frac{t^2}{1-t}+B(t,1)+R(t,1,1)\right).
\]
The first few terms of the series $PP_3(t)$ are
\[
\begin{split}
PP&_3(t)=6\,{t}^{2}+10\,{t}^{3}+24\,{t}^{4}+66\,{t}^{5}+198\,{t}^{6}+628\,{t}^{7}+2068\,{t}^{8}+7004\,{t}^{9}+24260\,{t}^{10}\\
&+85596\,{t}^{11}+306692\,{t}^{12}+1113204\,{t}^{13}+4085120\,{t}^{14}+15131436\,{t}^{15}+56495170\,{t}^{16}\\
&+212377850\,{t}^{17}+803094926\,{t}^{18}+3052424080\,{t}^{19}+11653580124\,{t}^{20}+\ldots.
\end{split}
\]

\section{Analytic properties of the generating functions}\label{Analyticproperties}
So far we have considered the generating functions in question as \emph{formal} power series. A crude estimate on the number of SAPs of half-perimeter $n$ is $4^{2n}$ which is the total number of all nearest neighbour walks on the square lattice of length $2n.$ So the series $PP_2(t)$ and $PP_{3}(t)$ converge at least in the open disc $\{|t|<1/16\}$ and represent analytic functions there. This section deals with the analytic properties of these functions.  
We first discuss the analytic structure of the generating function of two-sided PPs.
\begin{prop}The generating function $PP_2(t),$ cf. \eqref{gfP}, is algebraic of degree 2, with its dominant singularity a square root singularity at $t=\rho,$ where $\rho$ is the unique real root of the equation
\begin{equation}\label{eqrho}
\frac{1-4\,t+2\,{t}^{2}+{t}^{4}}{1-t}=1-3t-t^2-t^3=0.
\end{equation}
With $\theta=\sqrt[3]{26+6\sqrt{33}}$ the exact value for $\rho$ can be written as
\[
\rho=\frac{\theta^2-\theta-8}{3\theta}=0.2955977\ldots.
\]
The number $pp_2^{(m)}$ of two-sided PPs of half-perimeter $m$ is asymptotically
\[
pp_2^{(m)}\sim A\cdot \rho^{-m} \cdot m^{-3/2}\qquad(m\to\infty),
\]  
where 
\[
A=\frac{\sqrt{(-37+11\sqrt{33})\theta^2+(-152+8\sqrt{33})\theta+32}}{4\,\sqrt{6\pi}\rho}=0.8548166\ldots.
\]
\end{prop}
\noindent
\textbf{Remark:} \textit{i)} The generating function of two-sided prudent \emph{walks} is algebraic with its dominant singularity a simple pole at $\overline{\sigma}=0.403\ldots.$ Its coefficients are asymptotically equal to $\kappa\cdot\overline{\sigma}^{-m},$ where $\kappa=2.51\ldots,$ cf. \cite{Bou2}.\\
\textit{ii)} The asymptotic number of bar graphs as well as staircase polygons (counted by half-perimeter) and Dyck paths (by half-length) is of the form $\kappa\cdot\mu^n\cdot n^{-3/2}.$ Furthermore, the area random variables in the fixed-perimeter (fixed-length) ensembles of all three models are known to converge weakly to the \emph{Airy distribution} \cite{Duchon,Ric,Tak}.

\smallskip
The analytic structure of $PP_3(t)$ is far more complicated due to the analytic structure of $R(t,1,1),$ which is stated in the main result Theorem \ref{singstr}. In what follows we make frequent use of the following facts about the series $q:$
\begin{lemma}\label{Facts about Chuck Norris}
 The series $q,$ $(1-t)q-1,$ $q^2t,$ $t(1+t)q$ and $ t\left(1-{t} -{t}^{3} \right) q+{t}^{2}$ have non-negative integer coefficients. For $|t|\leq\rho$ we have the estimates
 \begin{equation}
 |q|\leq\frac{|t|^2+1}{2|t|},\;\left|q^2t\right|\leq1,\; \left|(1-t)q-1\right|\leq\rho,\;\left|1-t(1+t)q\right| \geq \rho.
 \end{equation}
 Equality holds if and only if $t=\rho.$ Furthermore 
 \begin{equation}
 q\left(\rho\right)=\frac{\rho^2+1}{2\rho}=\frac{1}{\sqrt{\rho}}.
 \end{equation}
\end{lemma}
\noindent
The singular behaviour of $B\left(t,q\left(q^2t\right)^N\right)$ and $B(t,qw)$ plays an important role in the study of $R(t,1,1).$
\begin{lemma}\label{Sigmas}
For $N\ge 0$ the dominant singularity of $B\left(t,q\left(q^2t \right)^N \right)$ is $\sigma_N,$ which is the unique solution in the interval $[0,\rho)$ of the equation
\[
u(t)-q\left(q^2t \right)^N=\frac{1}{t}\cdot \frac{1-\sqrt{t}}{1+\sqrt{t}}-q\left(q^2t \right)^N=0.
\]
In particular, $\sigma:=\sigma_0=\tau^2=0.2441312\ldots,$ where $\tau$ is the unique real root of the polynomial $t^5+2t^2+3t-2.$ The sequence $\left\{\sigma_N,\;N\ge 0\right\}$ is monotonically increasing and converges to $\rho.$ Furthermore $B(t,qw)$ is analytic in the polydisc $\{|t|<\rho\}\times \{|w|<\sqrt{\rho}\}.$
\end{lemma}
\begin{proof}
$B(t,u)$ is singular if and only if
\[
t^2(1-t)^2u^2-2t\left(1-t^2\right)u+(1-t)^2=0.
\]
The relevant solution $u(t)$ with $u(\rho)=1$ is
\begin{equation}\label{defu}
u(t)=\frac{1}{t}\cdot \frac{1-\sqrt{t}}{1+\sqrt{t}}.
\end{equation}
$B\left(t,q\left(q^2t\right)^N\right)$ is singular if $q\left(q^2t\right)^N=u(t).$ This equation has a solution $\sigma_N$ in the interval $(0,\rho),$ as $u(t)\rightarrow 1$ and $q\left(q^2t\right)^N\rightarrow (\rho^2+1)/2\rho=1/\sqrt{\rho}>1,$ for $t\rightarrow \rho.$ Here $u$ is  strictly decreasing and $q\left(q^2t\right)^N$ strictly increasing. We further see that $\sigma_N$ converges to $\rho,$ as for arbitrary fixed $t$ with $0<t<\rho$ we can chose $N$ sufficiently large, such that $u(t)> q\left(q^2t\right)^N,$ see Lemma \ref{Facts about Chuck Norris}. So $\sigma_N\geq t,$ which shows the convergence. Monotonicity follows, as $q\left(q^2t\right)^{N+1}<q\left(q^2t\right)^N$ for $t\in (0,\rho).$  All these singularities are square root singularities, as the expressions under the root are analytic in $|t|<\rho.$ $B(t,qw)$ is singular, if $w=u(t)/q$ and hence
\[
|w|=\frac{|u(t)|}{|q|}\geq \sqrt{\rho} u(\rho)=\sqrt{\rho},
\]
with equality if and only if $t=\rho.$ So there is no singularity inside the polydisc. 
\end{proof}
\noindent
Now we are ready to state the main result, which is proven in the subsequent lemmas.
\begin{theorem}\label{singstr}
The function $R(t,1,1)$ is analytic in the disc $\left\{|t|<\sigma \right\}$ with its unique dominant singularity a square root singularity at $\sigma.$ Moreover it is meromorphic in the slit disc 
\[
D_{\sigma,\rho}=\left\{|t|<\rho \right\}\setminus [\sigma,\rho),
\]
and it has infinitely many square root singularities in the set $\left\{\sigma_N,\,N=0,1,2,\ldots\right\}.$ 
In particular, $R(t,1,1)$ is not D-finite.
\end{theorem}
\noindent
\textbf{Remark.} \textit{i)} The number $pp_3^{(m)}$ of three-sided PPs of half perimeter $m$ is asymptotically equal to $\kappa \cdot \sigma^{-m} \cdot m^{-3/2}$ for some positive constant $\kappa.$ In particular, two-sided PPs are exponentially rare among three-sided PPs.\\ 
\textit{ii)} The generating function of three-sided prudent \emph{walks} has  its dominant singularity a simple pole at $\overline{\sigma}=0.403\ldots,$ as in the two-sided case. It is meromorphic in some larger disc of radius $\overline{\rho}=\sqrt{2}-1$ with infinitely many simple poles in the intervall $[\overline{\sigma},\overline{\rho}).$ Its coefficients grow like $\kappa \cdot\overline{\sigma}^{-m},$ for some $\kappa>0$ \cite{Bou2}. 

\smallskip
Possible singularities of $R(t,1,1)$ in $D_{\sigma,\rho}$ are zeroes of the denominators of $K(w)$ and $L(w),$ places, where the representation \eqref{gfR} diverges, and square root singularities of $B\left(t,q\left(q^2t \right)^N \right).$
Now we investigate the analytic properties of the single summands in the representation \eqref{gfR}.
\begin{lemma}\label{Lemma45}
\begin{enumerate}
\item $K\left(\left(q^2t\right)^N\right)$ and $L\left(\left(q^2t\right)^N\right)$ are analytic in $\{|t|<\sigma_N \}.$
\item $K\left(\left(q^2t\right)^Nw\right)$ and $L\left(\left(q^2t\right)^Nw\right)$ are analytic in $\{|t|<\rho \}\times \{|w|<\sqrt{\rho}\}.$
\end{enumerate}
\end{lemma}
\begin{proof}
With the above definition of $u(t)$ and a short computation we obtain the estimate 
\begin{equation}
\left| B\left(t,q\left(q^2t\right)^N\right)\right|<B(|t|,u(|t|))=\sqrt{|t|}.
\end{equation}
The denominator of $K(w)$ and $L(w)$ is
\[
1-T(t,w)=1-t \left( 1+t \right) q- \left(  t\left(1-{t} -{t}^{3} \right) q+{t}^{2} \right)  \left( B \left( t,qw \right) +t \right) w.
\]
$T(t,w)$ is a power series in $t$ and $w$ with non-negative coefficients and $T(0,w)=0.$ Hence we have the estimate
\[
T\left(t,\left(q^2t\right)^N\right)
\leq T\left(\sigma_N, \left(q\left(\sigma_N\right)^2\sigma_N\right)^N
\right)
\leq T\left(\sigma_N, \frac{u\left(\sigma_N\right)}{q\left(\sigma_N\right)}\right).
\]
A computation shows that the function $1-T\left(t,u(t)/q(t)\right)$ 
\[
1-T\left(t,\frac{u(t)}{q(t)}\right)=1-t \left( 1+t \right) q- \left(  t\left(1-{t} -{t}^{3} \right) q+{t}^{2} \right)  \left( \sqrt{t} +t \right) \frac{u(t)}{q(t)}
\]
has no zeroes in $[\sigma,\rho].$ This finishes the proof of the first assertion, as $K\left(\left(q^2t\right)^N\right)$ and $L\left(\left(q^2t\right)^N\right)$ do not have poles inside $\{|t|<\sigma_N\}.$ Furthermore, the denominator $1-T\left(t,\left(q^2t\right)^Nw\right)$  is analytic in the polydisc $\{|t|<\rho \} \times \{|w|<\sqrt{\rho}\},$ with the only singular point $(t,w)=(\rho,\sqrt{\rho})$ on its boundary. As above we see 
\[
\left|T\left(t,\left(q^2t\right)^Nw\right)\right|\leq T(|t|,|w|) \leq T\left(\rho,\frac{u(\rho)}{q(\rho)} \right) 
= T\left(\rho,\sqrt{\rho} \right),
\]
and hence the denominator is non-zero in the domain in question and $K\left(\left(q^2t\right)^Nw\right)$ and $L\left(\left(q^2t\right)^Nw\right)$ are both analytic in the polydisc.
\end{proof}
\begin{lemma}\label{convergence}
\begin{enumerate}
\item The series representation \eqref{gfR} of $R(t,1,1)$ is a series of algebraic functions, which converges compactly in the slit disc $D_{\sigma,\rho}=\{|t|<\rho\}\setminus [\sigma,\rho)$ to a meromorphic function. 
\item Furthermore the corresponding representation of $R(t,w,w)$ converges compactly in the polydisc $\{|t|<\rho \} \times \{|w|<\sqrt{\rho}\}$ to an analytic function. 
\item The Taylor expansion of $R(t,w,w)$ about $(t,w)=(0,0)$ converges absolutely in $\{|t|<\rho \} \times \{|w|<\sqrt{\rho}\}.$
\end{enumerate}
\end{lemma}

\begin{proof}
For the first assertion choose $0<r<\rho.$ We look at the disc $\{|t|\leq r\}.$ The term independent of $w$ in the numerator of $K(w)$ is strictly less than $\rho$ for $|t|\leq r$ and the corresponding term in the denominator is strictly larger than $\rho, $ see Lemma \ref{Facts about Chuck Norris}. So we can choose $N$ large such that $\sigma_N>r$ and $\left|K\left(\left(q^2t\right)^N\right)\right|<1$ for $|t|\leq r.$ Split the series at $N.$ The summands for $k=0,\ldots,N-1$ sum up to a function which is meromorphic in the slit disc $\left\{|t|\leq r \right\}\setminus [\sigma,r].$ In the rest of the series take out the common factors to obtain 
\begin{equation}\label{sumgen}
\prod_{j=0}^{N-1}K\left(\left(tq^2  \right)^{j}\right) \sum_{k\ge 0}L\left(\left(tq^2  \right)^{N+k}\right)\prod_{j=0}^{k-1}K\left(\left(tq^2  \right)^{N+j}\right).
\end{equation}
 The first product is a meromorphic function in the slit disc. $L\left(\left(tq^2  \right)^{N+k}\right)$ is easily seen to converge uniformly to $0$ in $|t|\leq r$ as $k\rightarrow \infty.$  In $|t|\leq r$ all summands are holomorphic (see the above discussion) and the sum can be estimated by a geometric series and hence converges uniformly in the compact disc  $\{|t|\leq r\}.$ By Montel's theorem the limit of the sum is again analytic. This finishes the proof for the first assertion. The second assertion is proven along the similar lines. By the multivariate version of Montel's theorem \cite{Sch} the limit function is also analytic in the domain in question and thus the third assertion follows. 
 \end{proof}

\begin{lemma}\label{Lemma47}
$R(t,1,1)$ is singular at infinitely many of the $\sigma_N.$ Furthermore, $R(t,1,1)$ is singular at $\sigma.$
\end{lemma}
\begin{proof}
Terms singular at $\sigma_N$ only show up in the summands for $k\ge N.$ The sum of these \eqref{sumgen} is equal to
\begin{equation}
\prod_{j=0}^{N-1}K\left(\left(tq^2  \right)^{j}\right) 
\left[ 
L\left(\left(tq^2  \right)^{N}\right)  
+K\left(\left(tq^2  \right)^{N}\right) R\left(t,\left(tq^2  \right)^{N+1},\left(tq^2  \right)^{N+1}\right)  
 \right].
\end{equation}
In order to show that the singularity $\sigma_N$ does not cancel, only the term in square brackets is of interest. Singular terms show up in the numerators and the common denominator of $K\left(\left(tq^2  \right)^{N}\right) $ and $L\left(\left(tq^2  \right)^{N}\right).$ We now manipulate the expressions \eqref{Kalt} and \eqref{Lalt} for $K(w)$ and $L(w)$ in order to get rid of singular terms in the denominator, where the factor 
\[
(1-q)qwt\left(B(t,qw)+t\right) +t
\]
leads to a singularity at $\sigma_N$ for $w=\left(tq^2  \right)^{N}.$ Write
\[
qwt\left(B(t,qw)+t\right)=A(w)-\phi(w),
\]
where
\[
A(w)=\frac{1}{2}\left( 1+t-qw(1+t)t \right)
\]
\[
\phi(w)=\frac{1}{2}\sqrt{t^2(1-t)^2(qw)^2-2t\left(1-t^2\right)qw+(1-t)^2}.
\]
Then $A\left(\left(tq^2  \right)^{N}\right)$ is analytic in $\{|t|<\rho\}.$ After multiplication of the numerator and denominator with $(1-q)A(w)+t\;+\; (1-q)\phi(w)$ there is no more occurence of $\phi$ in the denominator. We now have to collect the terms involving $\phi(w)$ in the numerators of $K(w)$ and $L(w).$ In the numerator of $K(w)$ the terms involving $\phi(w)$ sum up to 
\[
P_K(w)\phi(w):=t(q-1)^2(1-q^2t)\phi(w).
\]
The terms involving $\phi(w)$ in the numerator of $L(w)$ sum up to
\[
P_L(w)\phi(w):=(1-qt)(1-q^2t)(1-q)t\phi(w).
\]
So the singularity at $\sigma_N$ can only cancel if
\begin{equation}\label{cancel}
-\frac{P_L\left(\left(\sigma_Nq(\sigma_N)^2  \right)^{N}\right)}  
{P_K\left(\left(\sigma_Nq(\sigma_N)^2  \right)^{N}\right)}=R\left(\sigma_N,\left(\sigma_Nq(\sigma_N)^2  \right)^{N+1},\left(\sigma_Nq(\sigma_N)^2  \right)^{N+1}\right).  
\end{equation}
In order to prove that this equation can hold for at most finitely many of the $\sigma_N,$ we show that for $\sigma_N$ sufficiently close to $\rho$ the lhs of eq. \eqref{cancel} is strictly decreasing while the rhs is strictly increasing. Since $\left(\sigma_N\right)$ is monotonically increasing and converges to $\rho$ this will finish the proof. 
We first prove the assertion on the rhs. The Taylor expansion of $R(t,w,w)$ about $(0,0)$ has non-negative coefficients and represents $R(t,w,w)$ in the polydisc $\{|t|<\rho \} \times \{|w|<\sqrt{\rho}\}$ by Lemma \ref{convergence}. 
By the definition of $\sigma_N$ and $u(t)$ we have 
\[
\left(\sigma_N q\left(\sigma_N\right)^2  \right)^{N+1}=u\left(\sigma_N\right)q
\left(\sigma_N\right)\sigma_N
\]
The rhs of the last equation is strictly increasing for sufficiently large $N$ and converges to $\sqrt{\rho}$ as $N\rightarrow \infty.$ The sequence $\sigma_N$ is also strictly increasing by Lemma \ref{Sigmas}. So for large enough $N$ the sequence $R\left(\sigma_N,\left(\sigma_N q\left(\sigma_N\right)^2  \right)^{N+1},\left(\sigma_N q\left(\sigma_N\right)^2  \right)^{N+1}\right)$ is strictly increasing.\\
Now we turn to the lhs of eq. \eqref{cancel}.  A computation yields 
\[
-\frac{P_L\left(\left(\sigma_Nq(\sigma_N)^2  \right)^{N}\right)}  
{P_K\left(\left(\sigma_Nq(\sigma_N)^2  \right)^{N}\right)}=\frac{1-\sigma_Nq(\sigma_N)}{q(\sigma_N)-1},
\]
which easily seen to be ultimately strictly decreasing. This finishes the proof of Lemma \ref{Lemma47}.
\end{proof}
\noindent
The Lemmas \ref{Lemma45}, \ref{convergence} and \ref{Lemma47} together constitute a proof of Theorem \ref{singstr}.
\section{Random three-sided prudent polygons}\label{Random}

In \cite{Bou2} prudent walks of a given fixed length are generated uniformly at random with a refined version of a method proposed in \cite{Wilf}. We briefly describe a version of the method tailored to our particular needs. The main ingredient in the sampling procedure are \emph{generating trees}. These are rooted trees with their nodes labelled in such a way that if two nodes bear the same label, then the multisets of the labels of their children are the same. In this section we present generating trees and sampling procedures for the various classes of prudent polygons.

\smallskip
The decompositions underlying the functional equations \eqref{feqB}, \eqref{feqR} and \eqref{feq4sd} (see also Figures \ref{feq2sd}, \ref{3sdfeq} and \ref{4sd}) yield rules according to which a larger PP from the respective class can be constructed starting from a smaller one. We refine these building steps such that each step increases the half-perimeter by one. The result is a step-by-step procedure which allows to generate any PP of half-perimeter $m$ in a \emph{unique} way, starting from the unit square, such that after the $k$th step we have a PP of half-perimeter $k+2,$ $k=0,1,\ldots m-2.$ To put it differently, we can organise the polygons in a rooted tree $\mathcal T,$ with the unit square as the root and the polygons of half-perimeter $m$ as the nodes on level $m-2.$ So a random PP of half-perimeter $m$ corresponds to a random path of length $m-2$ in that tree starting from the root.
A polygon $\alpha$ is a \emph{child} of a polygon $\pi,$ if it is obtained by one of the following six construction steps.
\begin{enumerate}
\item Attaching a new top row which is shorter than or equal to the current top row,
\item attaching a unit square to the left side of the current top row,
\item attaching a new leftmost column which is shorter than or equal to the current leftmost column,
\item attaching a unit square to the bottom side of the leftmost column,   
\item attaching a new bottom-most row which is shorter than or equal to the current bottom-most row, and
\item attaching a unit square to the right side of the bottom-most row.   

\end{enumerate}
Any of these steps, if applicable, increases the half-perimeter by one, see Figure \ref{randsteps} below. \\

\begin{figure}[htb]
\begin{center}
\includegraphics[height=10mm,width=115mm]{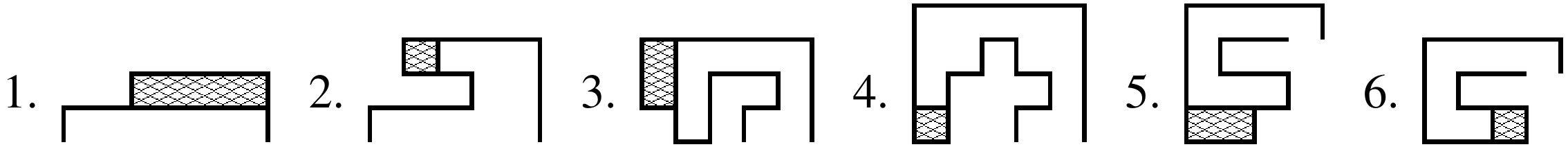}
\caption{\small The types of steps used to obtain generating trees}\label{randsteps}
\end{center}
\end{figure}
\noindent
\textbf{Remark.} \textit{i)} Steps of types 2,4 and 6 are only admissible if the current top row (leftmost column, bottom row) is longer than or equal to the second row from top (second column from the left, second row from the bottom). Additionally, a step of type 6 is forbidden, if the bottom row is only one unit shorter than the width of the box.\\
\textit{ii)} In the proof of the functional equation \eqref{feqR} the steps of types 3 and 4 are encapsulated in the ``attaching a bar graph to the left" operation. Hence any generic three-sided PP can be generated starting from the unit square by using only steps of the first four types.\\
\textit{iii)} Building bar graphs only requires steps of type 1 and 2. Here we reflected the bar graphs discussed earlier in the line $x=y.$

In order to compute the appropriate probabilities according to which each step in the random path in the tree is chosen, we associate to each PP a label encoding the admissible steps which can be applied to enlarge it. We give a labelling for unrestricted PPs, since the labellings of the other classes are obtained as specialisations thereof. The labels are five-tuples $(a,e,k,l,p)\in\{B,L,T\}\times \{\textnormal{y,n}\}\times\mathbb{Z}_{\ge 0}^3.$ The letter $a$ encodes the last building step. It is equal to $T$ (top), if the last step was of type 1 (which inflated the box to the top), or of type 2 but \emph{without} inflating the box  to the left. $a$ is equal to $L$ (left) if the last step was of type 2 \emph{and} thereby inflating the box to the left, of type 3 or of type 4 but \emph{without} inflating the box. Finally, $a$ is equal to $B$ (bottom) if the last building step was of type 4 \emph{with} an inflation of the box, or of types 5 or 6.

If $a=T,$ the parameter $e\in \{\textnormal{y,n}\}$ (yes/no) indicates if the current top row is longer than or equal to the second row from the top, and hence if a step of type 2 is applicable. Similarly, if $a=L,$ $e$ decides if a step of type 4 can be performed, i.e. if the leftmost column is shorter than or equal to the second but leftmost one. Finally, if $a=B,$ $e$ decides whether a step of type 6 can be performed. 

The parameter $k$ always denotes the length of the top row, and $l$ is either the distance of the left end of the top row to the left side of the box or the length of the leftmost column or the length of the bottom row, depending on whether $a=T$ or $a=L$ or $a=B$ respectively. The unit square receives the label $(L,\textnormal{n},1,1,0).$
\begin{center}
\begin{tabular}{|c|c|c|c|c|}\hline
\multicolumn{5}{|c|}{Labels for the generating tree of PPs}\\ \hline\hline
{$a$} & {$e$} & {$k$} & {$l$} & {$p$}\\ \hline\hline
\small{$T$} & \small{top row extendable?} & \small{top row} & \small{dist. of box to top row} & \small{height} \\ \hline
\small{$L$} & \small{left col. extendable?} & \small{top row} & \small{length of leftmost col.} & \small dist. of box to left col. \\ \hline
\small{$B$} & \small{bottom row extendable?} &\small{top row} & \small{length of bottom row} & \small{height} \\\hline
\end{tabular}
\end{center}
\noindent
\textbf{Remark.} In the proof of equation \eqref{feq4sd} we introduced two subclasses $\mathcal G$ and $\mathcal H.$ The polygons with $a=L$ (resp. $a=B$) are precisely those in $\mathcal G$ (resp. $\mathcal H$).

\smallskip 
The construction steps yield the following rewriting rules for the labels associated with general PPs. 
\begin{equation}\label{Tn}
(T,\textnormal{n},k,l,p)\to \begin{cases} (T,\textnormal{n},i,l+k-i,p+1),& i=1,\ldots, k-1\\  (T,\textnormal{y},k,l,p+1)& \end{cases}  
\end{equation}
%
%
\begin{equation}\label{Ty}
(T,\textnormal{y},k,l,p)\to \begin{cases} (T,\textnormal{n},i,l+k-i,p+1),& i=1,\ldots, k-1\\  (T,\textnormal{y},k,l,p+1)\\  (T,\textnormal{y},k+1,l-1,p),&\textnormal{if }l\ge 1\\  (L,\textnormal{n},k+1,1,p-1),&\textnormal{if }l=0   \end{cases}  
\end{equation}
%
%
\begin{equation}\label{Ln}
(L,\textnormal{n},k,l,p)\to \begin{cases} (T,\textnormal{n},i,k-i,l+p+1),& i=1,\ldots, k-1\\  (T,\textnormal{y},k,0,l+p+1)\\  (L,\textnormal{n},k+1,i,l+p-i),&\,i=1,\ldots, l-1\\  (L,\textnormal{y},k+1,l,p)& \end{cases}  
\end{equation}
\begin{equation}\label{Ly}
(L,\textnormal{y},k,l,p)\to \begin{cases} (T,\textnormal{n},i,k-i,l+p+1),& i=1,\ldots, k-1\\  (T,\textnormal{y},k,0,l+p+1)\\  (L,\textnormal{n},k+1,i,l+p-i),& i=1,\ldots, l-1\\  (L,\textnormal{y},k+1,l,p)&\\(L,\textnormal{y},k,l+1,p-1)&\textnormal{if }p\ge 1\\ (B,\textnormal{n},k,1,l+1)&\textnormal{if }p=0 \end{cases}  
\end{equation}
\begin{equation*}
(B,\textnormal{n},k,l,p)\to \begin{cases}   (T,\textnormal{n},i,k-i,p+1),& i=1,\ldots, k-1\\  (T,\textnormal{y},k,0,p+1)\\  (L,\textnormal{n},k+1,i,p-i),& i=1,\ldots, p-1\\  (L,\textnormal{y},k+1,l,0)\\ 
(B,\textnormal{n},k,i,p+1), & i=1,\ldots,l-1 \\ (B,\textnormal{y},k,l,p+1)  &\textnormal{if } k-1>l\\   (B,\textnormal{n},k,l,p+1)  &\textnormal{if } k-1=l         \end{cases}
\end{equation*}
\begin{equation*}
(B,\textnormal{y},k,l,p)\to \begin{cases}   (T,\textnormal{n},i,k-i,p+1),& i=1,\ldots, k-1\\  (T,\textnormal{y},k,0,p+1)\\  (L,\textnormal{n},k+1,i,p-i),& i=1,\ldots, p-1\\  (L,\textnormal{y},k+1,l,0)\\ 
(B,\textnormal{n},k,i,p+1), & i=1,\ldots,l-1 \\ (B,\textnormal{y},k,l,p+1)  &\\  (B,\textnormal{n},k,l+1,p)  &\textnormal{if } k-1=l \\ (B,\textnormal{y},k,l+1,p)  &\textnormal{if } k-1>l\\         \end{cases}
\end{equation*}

\noindent
The labelled rooted tree generated according to these rewriting rules with its root labelled $(L,\textnormal{n},1,1,0)$ is a generating tree for unrestricted prudent polygons (more precisely, for the class $\mathcal{F},$ cf. Section \ref{Functionalequations}). This tree is of course isomorphic to the tree $\mathcal T$ defined above having PPs as nodes, simply by replacing each PP by its label.

\smallskip 
As mentioned above, choosing a PP of half-perimeter $m$ uniformly at random is equivalent to choosing an $m-2$-step path starting from the root uniformly at random. This is achieved by picking each step in the path according to an appropriate probability which in turn can be expressed in terms of \emph{extension numbers}. If $\pi$ is a polygon of half-perimeter $m-s$ (a path of length $m-s-2$), then $EX(\pi,s)$ denotes the number of polygons of half-perimeter $m$ which can be reached from $\pi$ in $s$ construction steps, or equivalently of extensions of length $s$ of the path. Denote by $Ch(\pi)$ the set of polygons obtained from $\pi$ in one step, i.e. the children of $\pi$ in the generating tree. Now the right probability to choose $\alpha\in Ch(\pi)$ in the course of our random sampling procedure is equal to
\[
\mathbb{P}(\alpha|\pi)=\frac{EX(\alpha,s-1)}{EX(\pi,s)}.
\]   
The numbers $EX(\pi,s)$ can be computed recursively, namely
\[
EX(\pi,s)=\begin{cases} 1&\textnormal{if }s=0,\\ \sum_{\alpha\in Ch(\pi)}EX(\alpha,s-1)&\textnormal{otherwise}. \end{cases}
\] 
The crucial observation is that $EX(\pi,\cdot)$ \emph{only depends on the label of }$\pi,$ which allows an efficient computation. 

\smallskip
For unrestricted PPs, in the first $m-2$ levels of the tree $O(m^3)$ different labels occur since none of the parameters exceeds $m.$ It hence takes $O(m^4)$ operations to compute the all required extension numbers. We have implemented the procedure and computed these numbers up to $m=80.$ See Figure \ref{random4sd} at the end of the paper for some samples.

\medskip
\noindent
\textbf{Modifications for two-sided PPs.} As remarked above, generating two-sided PPs only requires steps of types 1 and 2. The only required information for the building procedure is the length of the top row and if the top row is extendable. We hence only need labels $(T,\textnormal{n},k)$ and $(T,\textnormal{y},k)$ obtained from the labels $(T,\cdot,k,l,p)$ above by leaving the parameters $l$ and $p$ unconsidered. The rewriting rule \eqref{Tn} can be adapted unchanged (up to deleting the last two coordinates) and in \eqref{Ty} simply omit the last line (and the ``if $l\ge1$"-clause in the second but last line). The unit square receives the label $(T,\textnormal{y},1).$ There are $O(m)$ different labels in the first $m-2$ levels of the generating tree, and hence $O(m^2)$ extension numbers have to be computed. See Figure \ref{random2sd} for some samples of half-perimeter $250.$

\medskip
\noindent
\textbf{Modifications for three-sided PPs.} For the generation of three-sided PPs steps 1, 2, 3 and 4 suffice. For an appropriate labelling we can hence dump down the $(B,\cdot,k,l,p)$ labels and use labels $(T,\textnormal{n},k,l),$ $(T,\textnormal{y},k,l),$ $(L,\textnormal{n},k,l)$ and $(L,\textnormal{y},k,l)$  obtained from the labels $(T,\cdot,k,l,p)$ and $(L,\cdot,k,l,p)$ by simply discarding the parameter $p.$ The rewriting rules \eqref{Tn}, \eqref{Ty} and \eqref{Ln} are adapted without change. In the rule \eqref{Ly} drop the last line. The unit square is labelled with $(L,\textnormal{n},1,1).$ We have $O(m^2)$ different labels on the first $m-2$ levels of the tree and hence $O(m^3)$ extension numbers have to be computed. See Figure \ref{random3sd} for some samples of half-perimeter $250.$
\section{Conclusion}\label{Conclusion}
We have solved the class of two-sided and three-sided prudent polygons, the generating function being algebraic in the former and non-D-finite in the latter case.
The analysis shows that two-sided PPs are exponentially rare among three-sided PPs which is different from the corresponding walk models where the growth rates are equal.\\
It would be nice to solve the class of general prudent polygons. We expect that the involved functional equations require three or more catalytic variables, which is the case for the equation found for the walk model.\\ 
Since the exponential growth rates of SAWs and SAPs are known to be equal \cite{Ham61} it is also interesting to compare the exponential growth rates of $k$-sided PWs and PPs. To that end it suffices to study PPs ending in $(1,0).$ As already mentioned in the introduction, a $k$-sided PP ending in $(1,0)$ may never step right of the line $x=1$ and it heads towards the vertex $(1,0)$ as soon as it hits that line for the first time in a point $(1,y_0).$ Up to that step the boundary walk of that $k$-sided PP is genuinely $k-1$-sided. This yields an injective map sending a $k$-sided PP to a $k-1$-sided PW simply by reflecting the segment joining $(1,y_0)$ and $(1,0)$ in the line $y=y_0,$ see Figure \ref{ktokminus1}.
\begin{figure}[htb]
\begin{center}
\includegraphics[height=26mm,width=60mm]{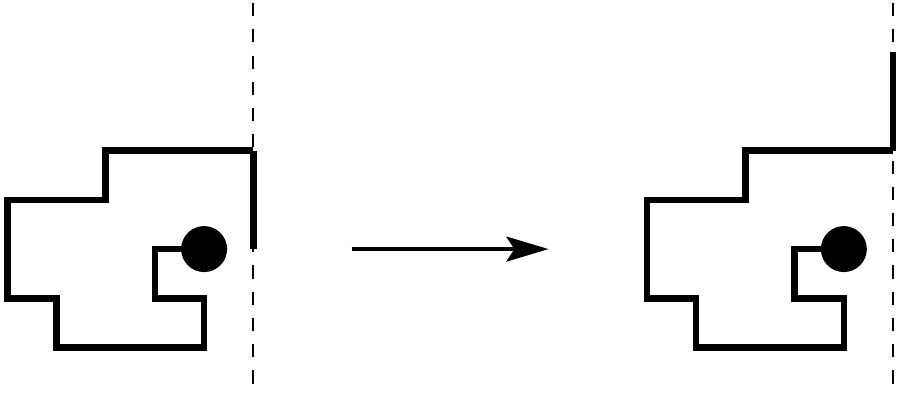}
\caption{\small Embedding of $k$-sided PPs into $k-1$-sided PWs}\label{ktokminus1}
\end{center}
\end{figure}
\noindent
We denote the so-obtained subclass of $k-1$-sided PWs by ``embedded $k$-sided PPs". If we count PPs by full perimeter, their exponential growth rates become $1/\sqrt{\rho}=1.83\ldots$ for two-sided PPs and $1/\sqrt{\sigma}=2.02\ldots$ for three-sided PPs. It is known that the exponential growth rate of PWs is equal to $1+\sqrt{2}=2.41\ldots$ in the one-sided case and equal to $2.48\ldots$ in the two- and three-sided cases \cite{Bou2}. The latter rate is also expected for unrestricted PWs \cite{DGJ,GGJD09}. Consequently, for $k=2,3,$ our results imply that $k$-sided PPs are exponentially rare among $k$-sided PWs \emph{and}, via embedding, among $k-1$-sided PWs. Furthermore, the rate of three-sided PPs is even smaller than that of one-sided PWs. This is not surprising looking at the pictures in Figure \ref{random3sd}, as such a PP roughly consists of two ``almost" one-sided PWs, one heading to the far left followed by one ``almost directed" walk up and to the right (and the closing tail). We expect that this heuristic argument also applies in the general case, which is also supported by an estimated value of approximately $2.1<1+\sqrt{2}$ for the growth rate of general PPs \cite{DGJ,GGJD09}.
\section*{Acknowledgements}
Part of this work was done during a stay of the author at the University of Melbourne. He would like to thank MASCOS and the members of the Statistical Mechanics and Combinatorics group for their kind hospitality. In particular he would like to thank Tony Guttmann for introducing the problem to him. He thanks Christoph Richard for helpful comments on the manuscript. He would also like to acknowledge financial support by the German Research Council (DFG) within CRC 701. 
\begin{figure}[!h]
\begin{center}
\includegraphics[height=70mm,width=76mm]{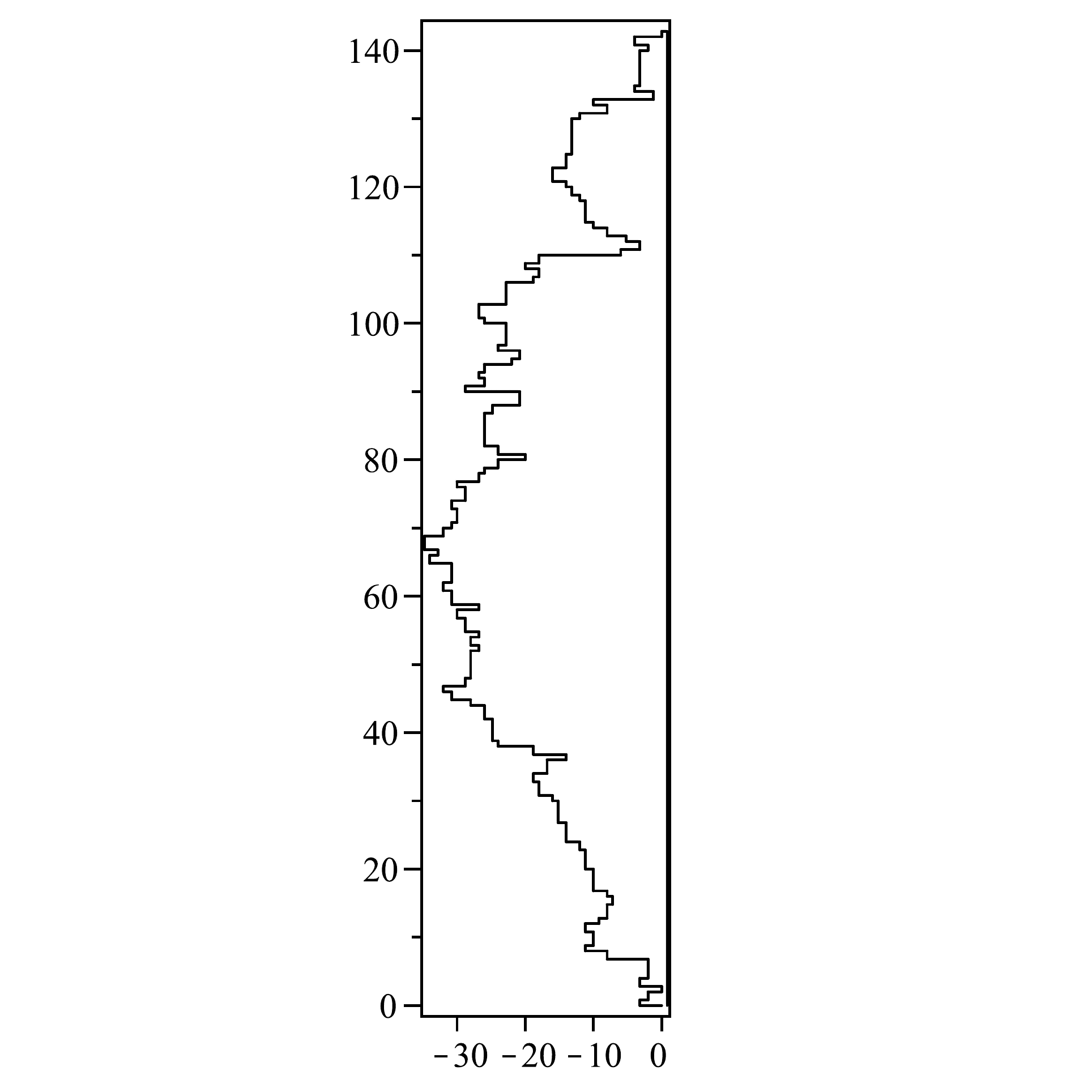}
\includegraphics[height=70mm,width=82mm]{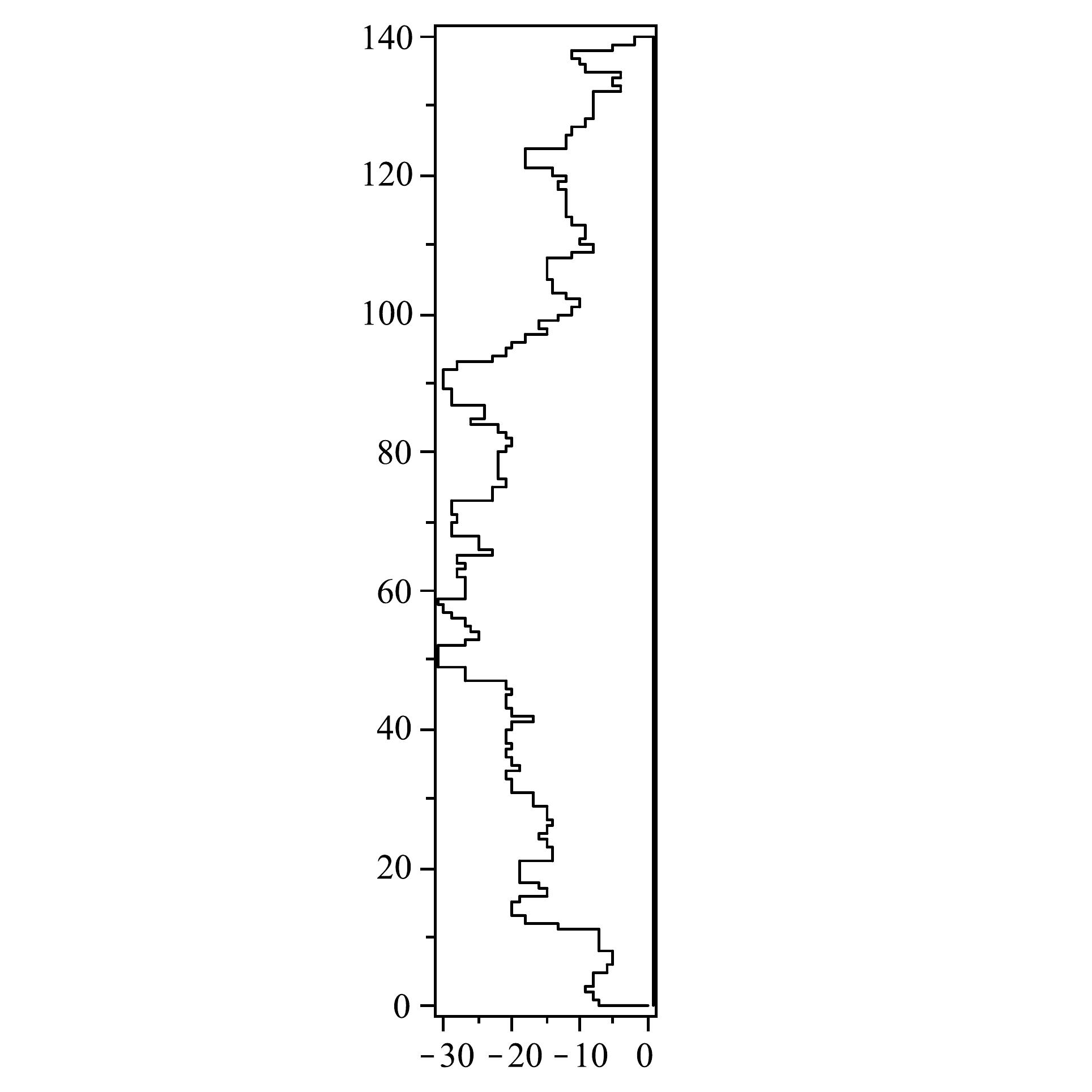}
\caption{\small Random 2-sided PPs of half-perimeter 250}\label{random2sd}
\end{center}
\end{figure}
\begin{figure}[!h]
\begin{center}
\includegraphics[height=43mm,width=63mm]{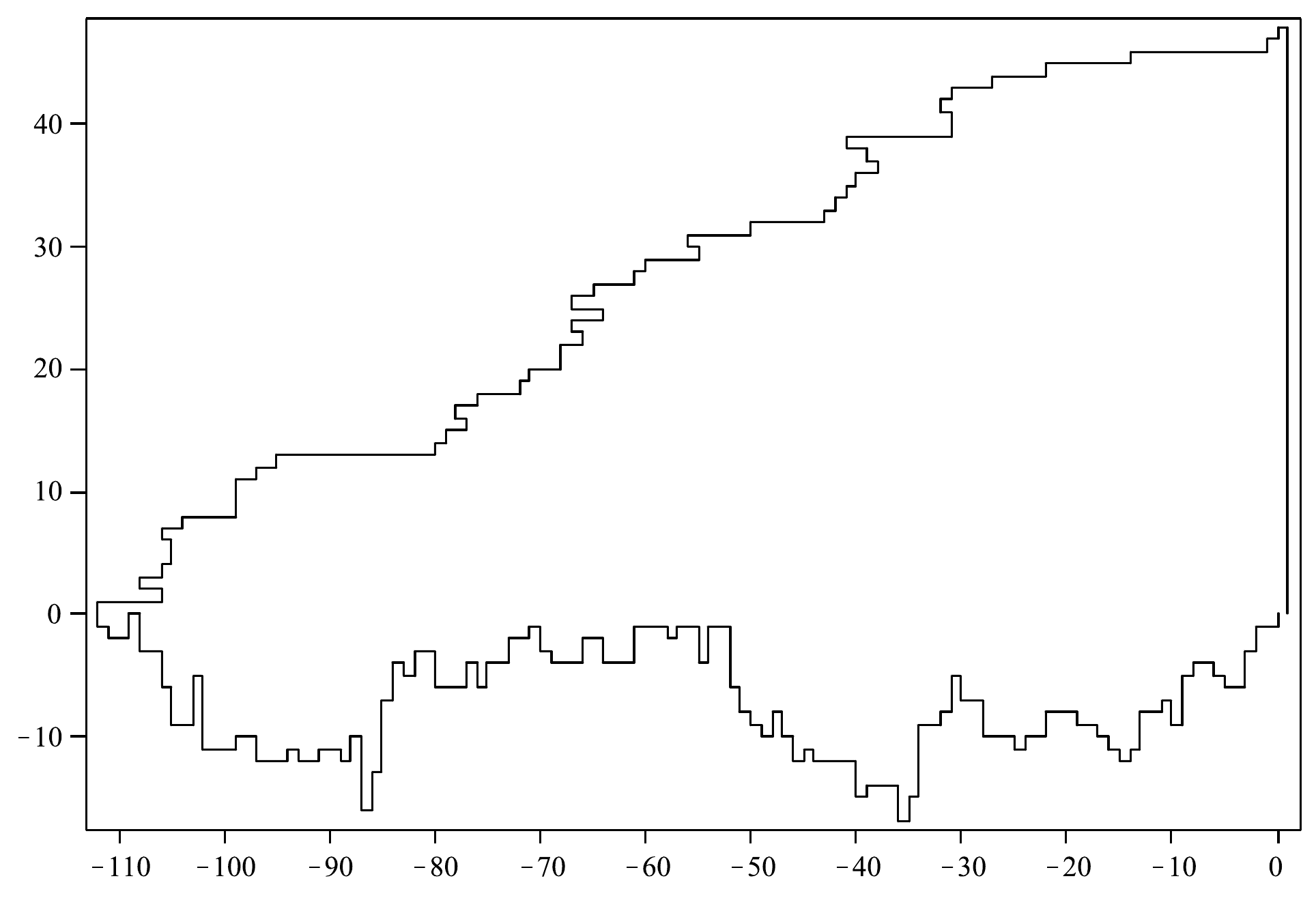}
\includegraphics[height=43mm,width=69mm]{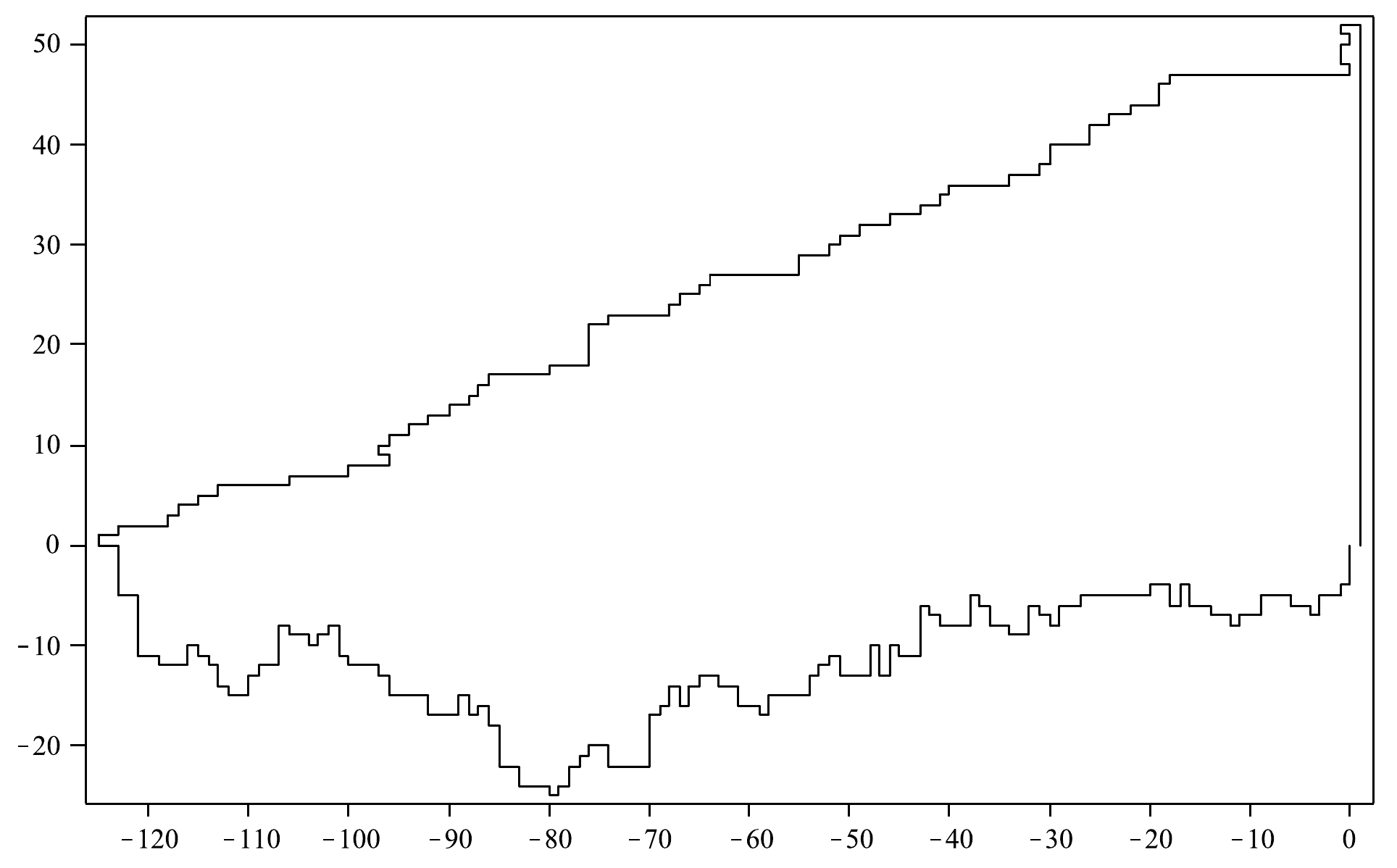}
\includegraphics[height=43mm,width=69mm]{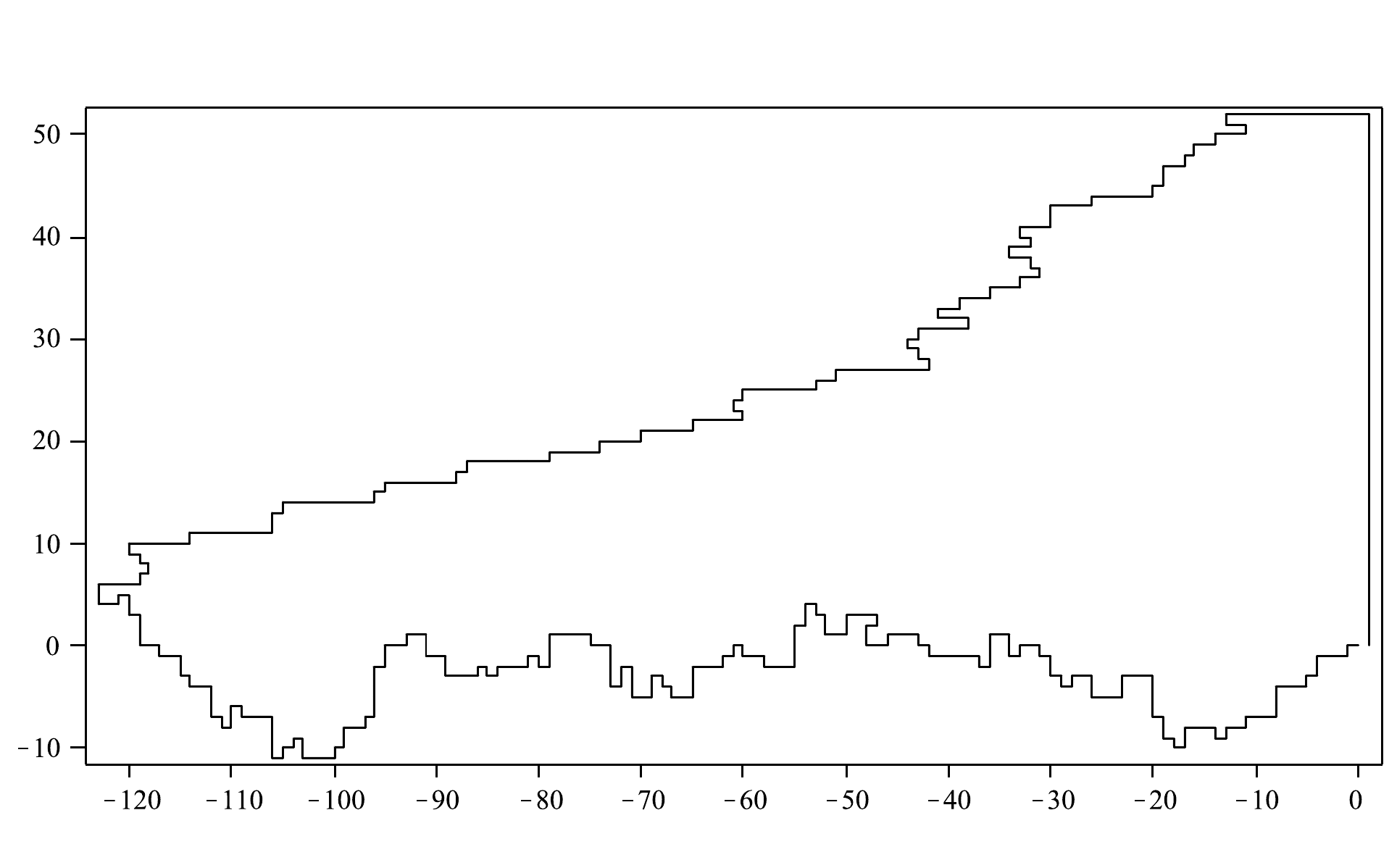}
\includegraphics[height=43mm,width=69mm]{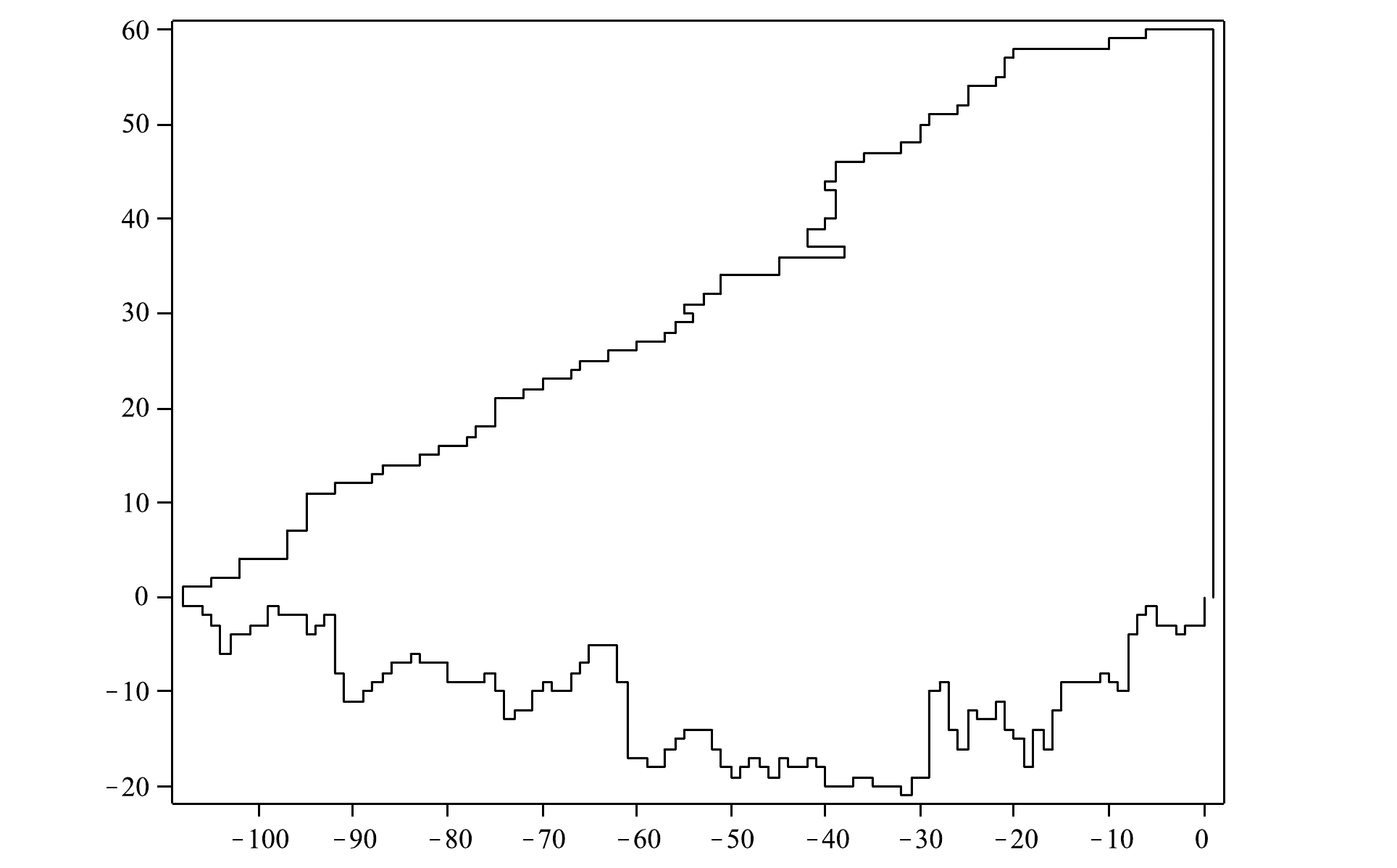}
\caption{\small Random 3-sided PPs of half-perimeter 250}\label{random3sd}
\end{center}
\end{figure}
\begin{figure}[!h]
\begin{center}
\includegraphics[height=63mm,width=63mm]{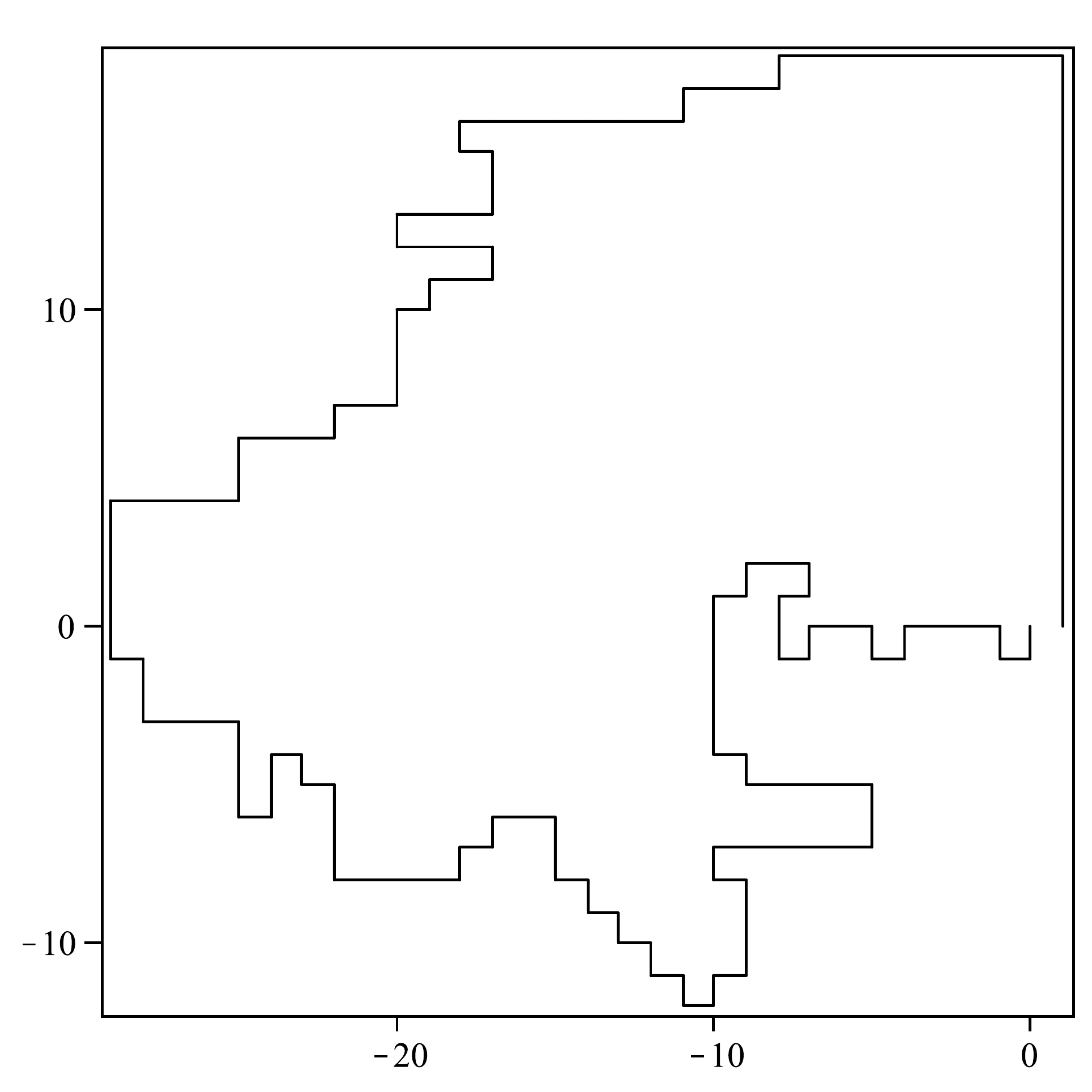}
\includegraphics[height=63mm,width=69mm]{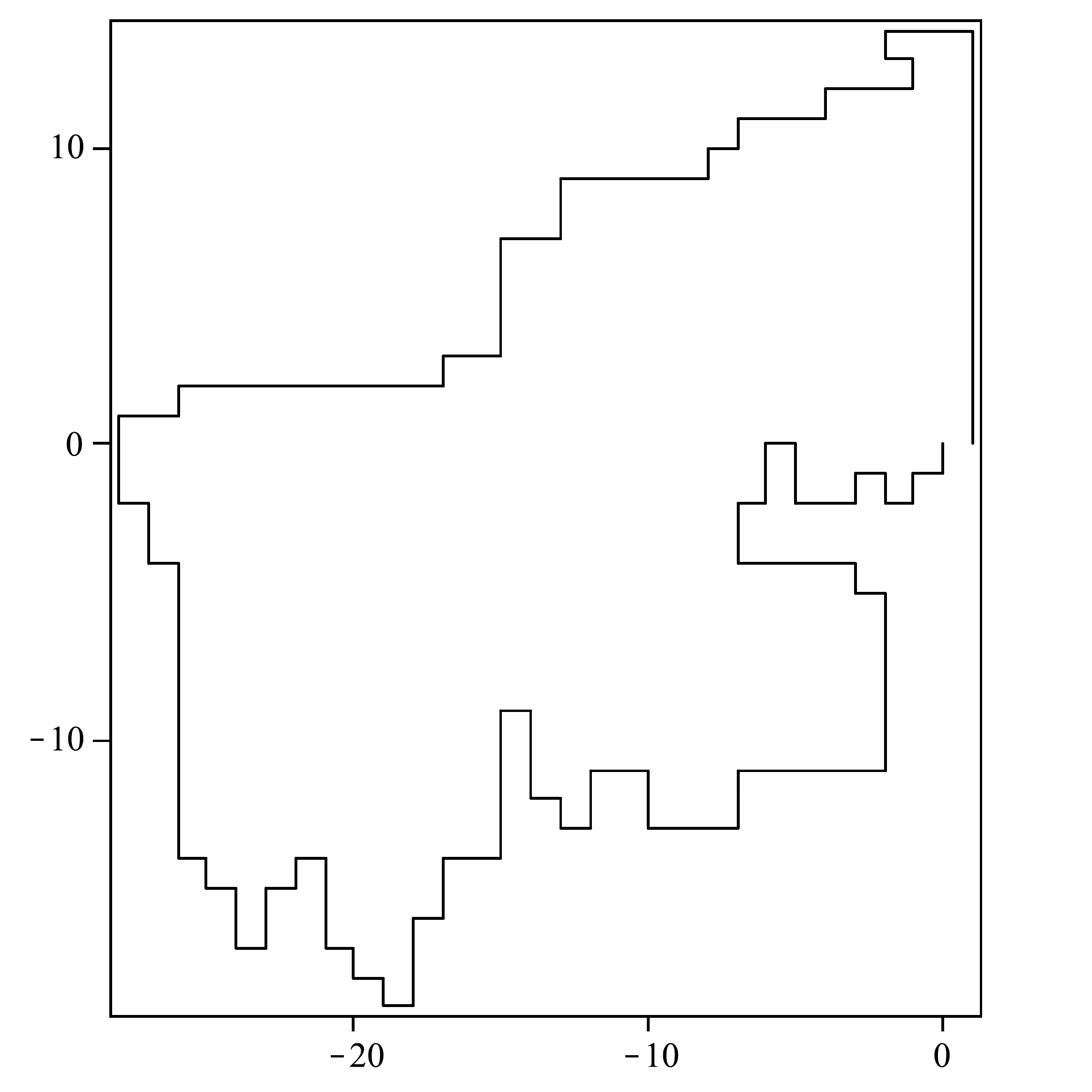}
\includegraphics[height=63mm,width=69mm]{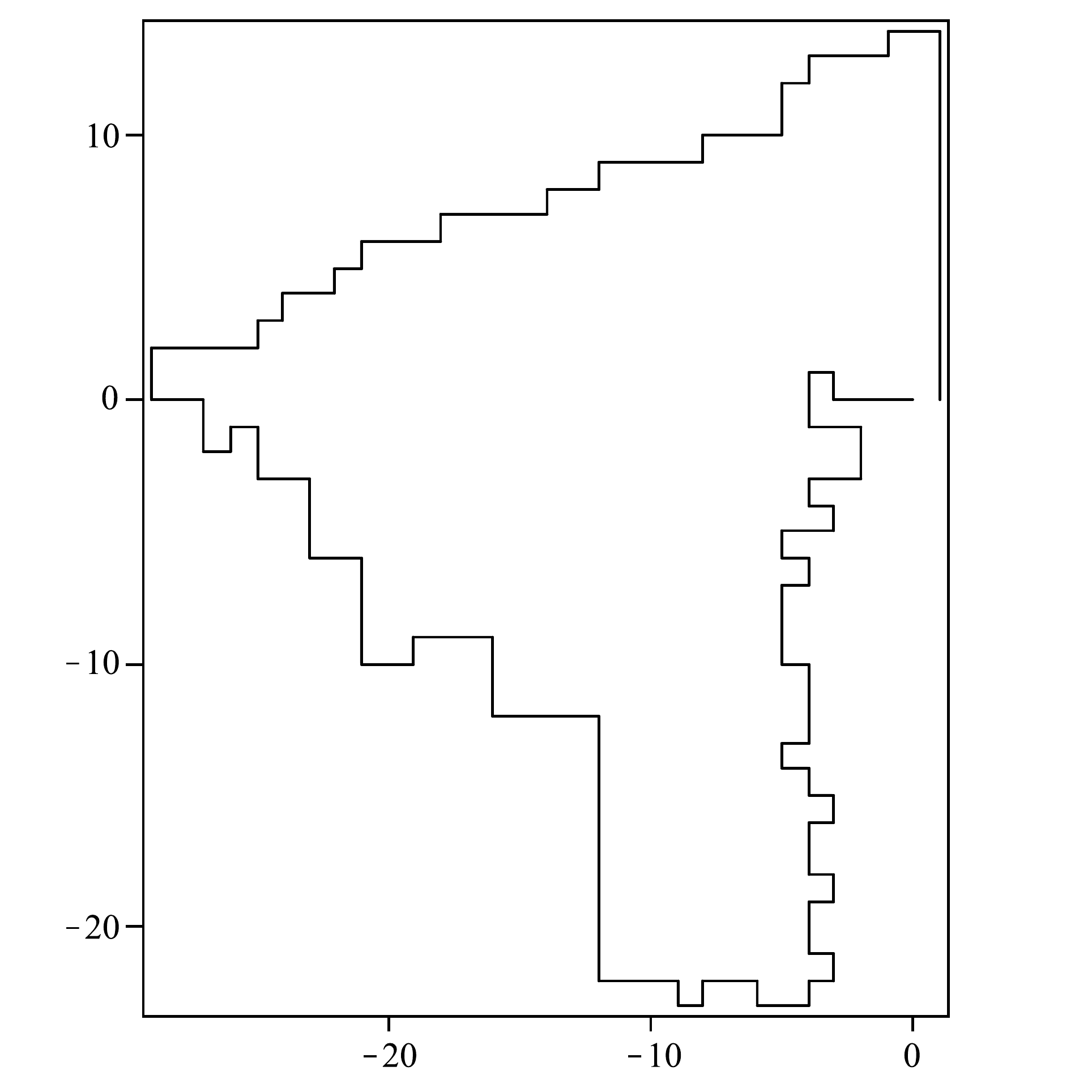}
\includegraphics[height=63mm,width=69mm]{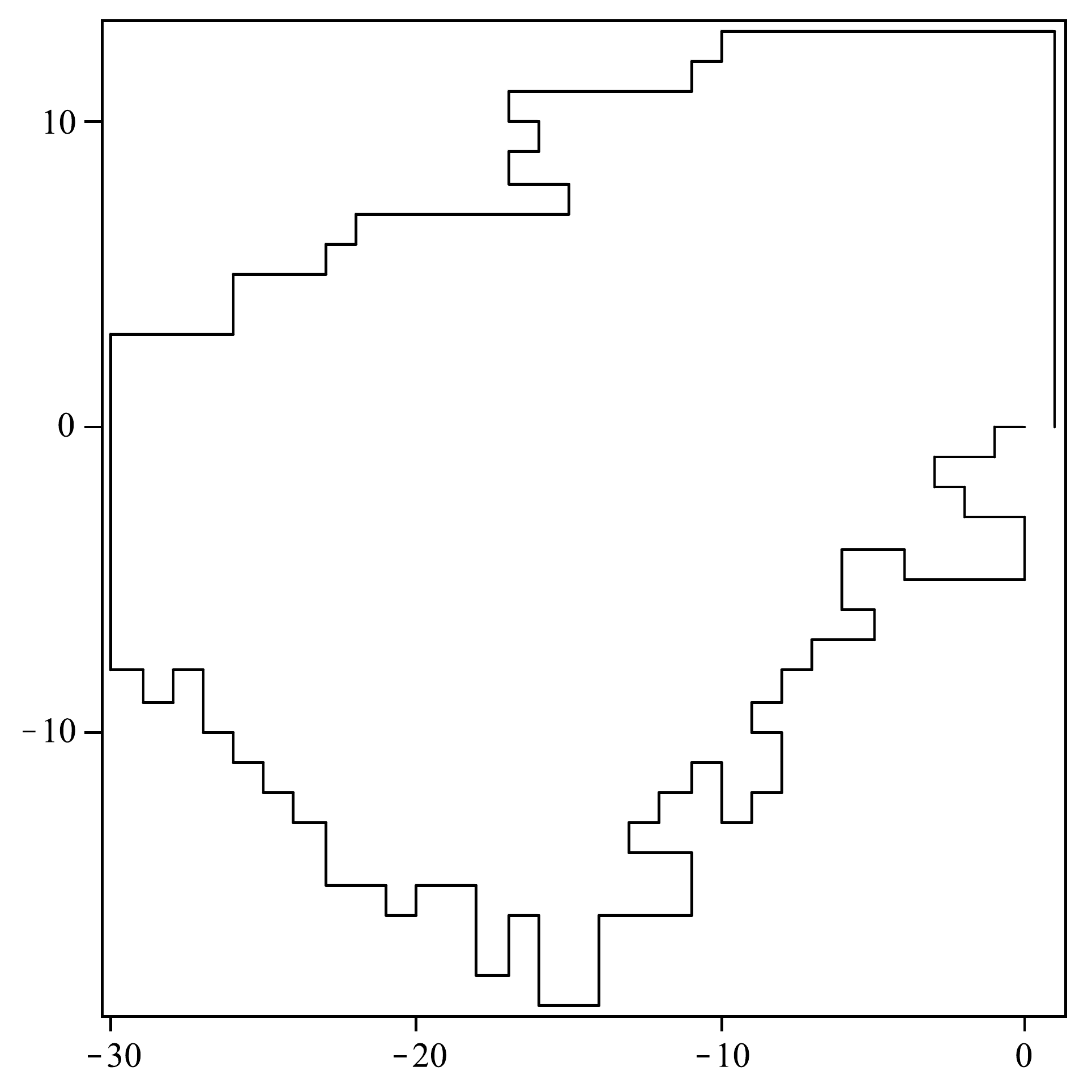}
\caption{\small Random unrestricted PPs of half-perimeter 80}\label{random4sd}
\end{center}
\end{figure}

\begin{thebibliography}{99}

\bibitem{Bou1}
M. Bousquet-M\'elou, A method for the enumeration of various classes of column-convex polygons, 
\textit{Discrete Math.} \textbf{154}(1996), 1-25, doi:10.1016/0012-365X(95)00003-F     

\bibitem{Bou2}
M. Bousquet-M\'elou, Families of prudent self-avoiding walks (2008), arXiv: math.CO 0804.4843.

\bibitem{BouJeh}
M. Bousquet-M\'{e}lou and A. Jehanne, Polynomial equations with one catalytic variable, algebraic series and map enumeration, \textit{J. Combin. Theory, Ser. B} \textbf{96}(2006), No.5, 623-672, doi:10.1016/j.jctb.2005.12.003

\bibitem{DGJ}
J. Dethridge, A.J. Guttmann and I. Jensen, Prudent self-avoiding walks and polygons. In \textit{Random Polymers}, EURANDOM, Eindhoven, The Netherlands (June 2007).

\bibitem{Duc}
E. Duchi, On some classes of prudent walks. In \textit{FPSAC'05}, Taormina, Italy (2005).

\bibitem{Duchon}
P: Duchon, $Q$-grammars and wall polyominoes, \textit{Ann. Comb.} \textbf{3}(1999), No. 2-4, 311-321, doi:10.1007/BF01608790

\bibitem{GGJD09}
T.M. Garoni, Timothy M., A.J. Guttmann, I. Jensen and J. Dethridge, Prudent walks and polygons, 
\textit{J. Phys. A} \textbf{42}(2009), No. 9,
doi: 10.1088/1751-8113/42/9/095205 



\bibitem{GouJac}
I.P. Goulden and D.M. Jackson, \textit{Combinatorial Enumeration},  New York, Wiley (1983).

\bibitem{AJG}
A.J. Guttmann, Some solvable, and as yet unsolvable, polygon and walk models, \textit{IOP, J. Phys.: Conference Series} \textbf{42}(2006), 98-110, doi: 10.1088/1742-6596/42/1/011


\bibitem{Ham61}
J.M. Hammersley, The number of polygons on a lattice, \textit{Proc. Camb. Philos. Soc.} \textbf{57} (1961), 516-523. 

\bibitem{MaSl}
N. Madras and G. Slade, \textit{The self-avoiding walk},
Probability and Its Applications. Boston, MA: Birkh\"auser. xiv, 425 p. (1993).

\bibitem{MisRec}
M. Mishna and A. Rechnitzer, Two non-holonomic lattice walks in the quarter plane, \textit{Theoret. Comput. Sci.} \textbf{410}(2009), No. 38-40, 3616-3630, doi:10.1016/j.tcs.2009.04.008  

\bibitem{Wilf}
A. Nijenhuis and H. Wilf, \textit{Combinatorial algorithms for computers and calculators. 2nd ed.}, Computer Science and Applied Mathematics. New York -San Francisco -
    London: Academic Press. xv, 302 p. (1978)

\bibitem{Pre}
P. Pr\'ea, Exterior self-avoiding walks on the square-lattice, unpublished manuscript (1997).

\bibitem{PreBra}
T. Prellberg and R. Brak, Critical exponents from nonlinear functional equations for partially directed cluster models, \textit{J. Stat. Phys.} \textbf{78}(1995), No. 3-4, 701-730, doi:10.1007/BF02183685

\bibitem{AR}
A. Rechnitzer, Haruspicy 2: The anisotropic generating function of self-avoiding polygons is not D-finite, \textit{J. Combin. Theory, Ser. A} \textbf{113}(2006), Issue 3, 520-546, doi:10.1016/j.jcta.2005.04.010 

\bibitem{Ric}
C. Richard, Limit distributions and scaling functions. 
In \textit{Polygons, Polyominoes and Polycubes}, 247-299, Lecture Notes in Physics. Berlin/Heidelberg: Springer (2009) 
doi:10.1007/978-1-4020-9927-4\_11  

\bibitem{Sch}
V. Scheidemann, \textit{Introduction to complex analysis in several variables}, Birkh\"auser, Basel (2005).

\bibitem{Tak}
L. Tak\'acs, A Bernoulli excursion and its various applications. \textit{Adv. Appl. Probab. }\textbf{23}(1991), Issue 3, 557-585, doi:10.2307/1427622

\end{thebibliography}
\end{document}